\documentclass[leqno]{amsart}
\usepackage{euscript, amssymb, amsfonts, color, xypic}

\include{pst-plot}
\usepackage{pstricks}

\newtheorem{theorem}{Theorem}[section]
\newtheorem{lemma}[theorem]{Lemma}
\newtheorem{corollary}[theorem]{Corollary}
\newtheorem{proposition}[theorem]{Proposition}
\newtheorem{notation}[theorem]{Notation}

\theoremstyle{definition}
\newtheorem{definition}[theorem]{Definition}

\theoremstyle{remark}
\newtheorem{remark}[theorem]{Remark}

\numberwithin{subsection}{section}
\numberwithin{equation}{section}
\numberwithin{figure}{section}

\newcommand{\Z}{\mathbb{Z}}
\newcommand{\C}{\mathbb{C}}
\newcommand{\AAA}{\mathbb{A}}
\newcommand{\M}{\mathbb{M}}
\newcommand{\F}{\mathbb{F}}

\newcommand{\p}{\mathbb{P}}
\newcommand{\E}{\mathcal{E}}

\DeclareMathOperator{\Ext}{Ext}
\DeclareMathOperator{\Tor}{Tor}
\DeclareMathOperator{\Spec}{Spec}
\DeclareMathOperator{\colim}{colim}

\DeclareMathOperator{\Map}{\mathrm{Map}}

\newcommand{\id}{\mathrm{id}}

\newcommand{\Sq}[1]{\mathsf{Sq}^{#1}}
\newcommand{\SQ}[1]{\mathsf{Q}_{#1}}
\newcommand{\SA}{\mathsf{A}}
\newcommand{\SE}{\mathsf{E}}

\newcommand{\map}{\rightarrow}

\newcommand{\Smash}{\wedge}

\newcommand{\bc}{\circle*{0.2}}
\newcommand{\rc}{\circle{0.3}}

\newcommand{\mhz}{{\line(0,1){1}}}
\newcommand{\mhzt}{{{\qbezier[5](0,0)(0,0.5)(0,1)}}}
\newcommand{\mhone}{{\line(1,1){1}}}

\psset{dash=3pt 3pt}

\begin{document}

\title{Motivic connective $K$-theories and the cohomology of $\SA(1)$}

\author{Daniel C.\ Isaksen}
\author{Armira Shkembi}

\address{Department of Mathematics\\ Wayne State University\\
Detroit, MI 48202}

\address{Mathematics and Sciences Department\\ St.\ Leo University\\ 
St.\ Leo, FL 33574}

\thanks{The first author was supported by NSF grant DMS0803997.}

\email{isaksen@math.wayne.edu}
\email{armira.shkembi@saintleo.edu}

\begin{abstract}
We make some computations in stable motivic homotopy theory
over $\Spec \C$, completed at 2.  
Using homotopy fixed points and the algebraic $K$-theory spectrum,
we construct a motivic analogue of the real $K$-theory spectrum $KO$.
We also establish a theory of connective covers to obtain a motivic
version of $ko$.
We establish an Adams spectral sequence for computing motivic $ko$-homology.
The $E_2$-term of this spectral sequence involves $\Ext$ groups
over the subalgebra $\SA(1)$ of the motivic Steenrod algebra.
We make several explicit computations of these $E_2$-terms in
interesting special cases.
\end{abstract}

\keywords{Motivic stable homotopy theory,
Adams spectral sequence,
connective $K$-theory,
Steenrod algebra}

\subjclass[2000]{14F42, 55N15, 55T15}

\maketitle

\section{Introduction}\label{sec:Mot}

In classical homotopy theory, there is a strong and useful relationship
between the connective $K$-theories $ku$ and $ko$ and algebraic computations
involving the subalgebras $E(1)$ and $A(1)$ of the Steenrod algebra.
More specifically, the $\F_2$-cohomology $H^*(ku)$ of $ku$ is equal
to the quotient $A//E(1)$ of the Steenrod algebra by the 
augmentation ideal of the subalgebra $E(1)$.
For formal algebraic reasons, it follows that the Adams spectral
sequence for computing $ku_*(X)$ has 
$\Ext_{E(1)} (H^*X, \F_2)$ as its $E_2$-term.
Thus, computations in $ku$-homology are essentially the same as algebraic
computations of $\Ext$ groups in the category of $E(1)$-modules.
Similarly, $H^*(ko)$ is equal to $A//A(1)$, and the Adams spectral
sequence for computing $ko_*(X)$ has
$\Ext_{A(1)} (H^*X, \F_2)$ as its $E_2$-term.

The goal of this paper is to describe similar phenomena in 
2-complete motivic stable homotopy theory over $\Spec \C$.
Over $\Spec \C$, the motivic Steenrod algebra is completely
understood \cite{V1}, and detailed algebraic computations are
possible \cite{DI3}.  After 2-completion, we know enough about
the motivic homotopy groups of spheres \cite{DI3} \cite{HKO} in order
to carry out the necessary homotopical arguments.

In motivic homotopy theory, the algebraic $K$-theory spectrum
$KGL$ plays the role of the classical complex $K$-theory spectrum $KU$.
We are able to construct a connective version, which we call $kgl$,
and we compute that the motivic cohomology $H^{*,*}(kgl)$
of $kgl$ with $\F_2$-coefficients is equal to $\SA//\SE(1)$, where
$\SA$ is the motivic Steenrod algebra and $\SE(1)$ is the subalgebra generated
by classes $\SQ{0} = \Sq{1}$ and 
$\SQ{1} = \Sq{1} \Sq{2} + \Sq{2} \Sq{1}$ of degrees $(1,0)$ and $(3,1)$ 
respectively.
It follows that over $\Spec \C$ and working $2$-complete,
the motivic Adams spectral sequence \cite{HKO} \cite{DI3}
for computing $kgl_{*,*}(X)$ has $\Ext_{\SE(1)} (H^{*,*} X, \M_2)$
as its $E_2$-term, where $\M_2 = \F_2[\tau]$ is the 
motivic cohomology of a point.

There are two possible approaches to constructing a motivic spectrum
that is analogous to classical $KO$.  The first approach is to 
use Hermitian $K$-theory, as in \cite{Hornbostel}, but we do not adopt
this viewpoint.  Rather, we use a $\Z/2$-action on $KGL$ and
apply homotopy fixed points to obtain a new motivic spectrum
$KGL^{h\Z/2}$.
This is analogous to the classical fact that $KO$ is equivalent
to $KU^{h\Z/2}$.
We shall call this spectrum $KO$ for convenience.

Although we believe that $KGL^{h\Z/2}$ is equivalent to Hermitian
$K$-theory over $\Spec \C$, we have not proved this result here
because the methods would take us too far afield from our main point.
Philosophically, we find that the homotopy fixed points description
of our motivic spectrum tells us everything that we would like to know 
about it.  It turns out that we really don't need a more concrete
geometric description.

Having constructed the motivic version of $KO$, we can consider
the connective cover $ko$.  We compute that the motivic cohomology
$H^{*,*}(ko)$ of $ko$ is equal to $\SA//\SA(1)$, where $\SA(1)$
is the subalgebra of the motivic Steenrod algebra generated by 
$\Sq{1}$ and $\Sq{2}$.
It follows that over $\Spec \C$ and working $2$-complete,
the motivic Adams spectral sequence
for computing $ko_{*,*}(X)$ has $\Ext_{\SA(1)} (H^{*,*} X, \M_2)$
as its $E_2$-term.

The result is that motivic $ko$-homology is effectively computable.
To demonstrate this point, we
compute the groups $\Ext_{\SA(1)}(M,\M_2)$ for various $\SA(1)$-modules $M$ 
of interest.  In this paper, we do not provide any calculations over $\SE(1)$
because the motivic calculations are essentially identical to the classical ones.

Our algebraic computations over motivic $\SA(1)$ are similar to the classical
computations over $\SA(1)$. 
One of the fundamental differences between classical homotopy theory
and motivic homotopy theory is that $\eta^4$ is zero classically,
but $\eta^k$ is non-zero motivically for all $k \geq 0$.
Our calculations over $\SA(1)$ detect this difference.

Along the way to computing the cohomology of $kgl$ and $ko$, we need a
collection of technical results about cellular motivic spectra over $\Spec \C$.
These results are likely to be useful in other contexts as well.

Ultimately, one would like to compute as much as possible
about $\Ext_A(\M_2,\M_2)$, which is the
$E_2$-term of the Adams spectral sequence that converges to motivic
stable homotopy groups.  Some progress on this has been made \cite{DI3},
and our computations over $\SA(1)$ also contribute to this larger program.
We have focused on $kgl$-homology and $ko$-homology because of their
relationship to motivic stable homotopy groups, but it is
possible to study $kgl$-cohomology and $ko$-cohomology within the
same framework.

Since we are working over $\Spec \C$,
there is a realization functor from motivic homotopy theory to
classical homotopy theory \cite{DDDI3} \cite[Section 3.3]{MV}.
Our motivic computations must be compatible under this functor with
the corresponding classical computations.  We strive to provide proofs
that are internal to motivic homotopy theory wherever possible,
but realization will be a useful tool for us from time to time.

In this paper, we will use the notion of motivic ring spectra, but
only in the naive sense.
For our purposes, we have no need for modern theories of
highly structured ring spectra, although such theories do exist for
motivic spectra \cite{DRO} \cite{Hovey} \cite{Hu} \cite{Jardine}.

\subsection{Organization}

We begin with a review of the algebraic objects that we will study.
Then we provide some background on 2-complete motivic stable homotopy theory
over $\Spec \C$.  Here we are collecting the homotopical tools necessary
for our later calculations.  We then compute the homotopy of $ko$
using a homotopy fixed points spectral sequence.  We also compute
the homotopy of $kgl^{h\Z/2}$; this calculation contains a curious
difference to the analogous classical calculation of the 
homotopy of $ku^{h\Z/2}$.  Our next calculation is the ordinary
motivic $\F_2$-cohomology of $kgl$ and $ko$.  This allows us to
describe the $E_2$-terms of Adams spectral sequences for computing
$kgl_{*,*}(X)$ and $ko_{*,*}(X)$.  Finally, we conclude the paper with
some specific calculations related to $ko$-homology.

\subsection{Acknowledgments}
We acknowledge many useful conversations with Robert Bruner,
Mark Behrens, and Mike Hill.  We particularly thank Paul Arne \O stv\ae r
for the structure of the argument in Section \ref{sctn:cohlgy}.

\section{Algebraic Definitions}
\label{sctn:intro}

In this section, we introduce the basic algebraic objects that we will study,
and we remind the reader of their relevance to motivic homotopy theory.

Throughout the paper, we will primarily be working with bigraded objects.

\begin{definition} An element of bidegree $(a,b)$ 
is said to have topological degree $a$ and weight $b.$
\end{definition}

\begin{definition} 
Let $\M_2$ be the bigraded polynomial ring $\F_2[\tau]$, 
where $\tau$ has bidegree $(0,1).$
\end{definition}

The importance of $\M_2$ is that it is the $\F_2$-motivic cohomology
ring of $\Spec \C$ \cite{V3}.

Recall that the motivic Steenrod algebra $\SA$ is an 
$\M_2$-algebra generated by elements $\Sq{2k}$ and $\Sq{2k+1}$
of bidegrees $(2k,k)$ and $(2k+1,k)$, subject to a motivic
version of the Adem relations \cite{V2}.

\begin{notation}
We write $\Sq{i_1,\ldots,i_n}$ for the product 
$\Sq{i_1}\cdots \Sq{i_n}.$
\end{notation}

\begin{definition}\mbox{}
\begin{enumerate}
\item
Let $\SE(0)$ be the $\M_2$-subalgebra of $\SA$ generated by $\Sq{1}$.
\item
Let $\SE(1)$ be the $\M_2$-subalgebra of $\SA$ generated by $\SQ{0} = \Sq{1}$
and $\SQ{1} = \Sq{1,2} + \Sq{2,1}$.
\item
Let $\SA(1)$ be the $\M_2$-subalgebra of $\SA$ 
generated by $\Sq{1}$ and $\Sq{2}$.
\end{enumerate}
\end{definition}

The following lemma is a straightforward calculation with low-dimensional
Adem relations.

\begin{lemma}\mbox{}
\begin{enumerate}
\item
$\SE(0)$ is equal to $\M_2[\Sq{1}] / \Sq{1,1}$.
\item
$\SE(1)$ is equal to $
\M_2[\SQ{0}, \SQ{1}] / \SQ{0}^2, \SQ{1}^2, \SQ{0}\SQ{1} + \SQ{1}\SQ{0}$,
i.e., the exterior $\M_2$-algebra on $\SQ{0}$ and $\SQ{1}$.
\item
$\SA(1)$ is equal to 
\[
\frac{\M _2[\Sq{1}, \Sq{2}]}{\Sq{1,1}=0, \Sq{2,2}= \tau \Sq{1,2,1}, \Sq{1,2,1,2}=\Sq{2,1,2,1}}
\] 
\end{enumerate}
\end{lemma}

\begin{remark}
Classically, the Adem relation $\Sq{2,2} = \Sq{1,2,1}$ implies that
$\Sq{1,2,1,2} = \Sq{2,1,2,1}$.  However, in the motivic situation,
the Adem relation $\Sq{2,2} = \tau \Sq{1,2,1}$ only implies that
$\tau \Sq{1,2,1,2} = \tau \Sq{2,1,2,1}$.  Therefore, we must include
the relation $\Sq{1,2,1,2} = \Sq{2,1,2,1}$ in the description of $\SA(1)$.
\end{remark}

Figure \ref{fig:A(1)} is a pictorial representation of $\SA(1)$,
where each circle at height $a$ stands for a copy of $\M_2$ in
topological degree $a$.  Multiplications by $\Sq{1}$
are represented by straight lines, and multiplications by $\Sq{2}$
are represented by curved lines.  The dashed line indicates
that $\Sq{2}$ on the generator in bidegree $(2,1)$ equals
$\tau$ times the generator in bidegree $(4,1)$.

For a subalgebra $B$ of the motivic Steenrod algebra $\SA$, we write
$\SA//B$ for the quotient of $\SA$ by the augmentation ideal of $B$.
In this paper, $B$ will be $\SE(0)$, $\SE(1)$, or $\SA(1)$.

\newpage

\begin{figure}[htbp!]
\begin{center}

\psset{unit=0.5cm}
\begin{pspicture}(-1,0)(1,7)

\pscircle(0,0){0.3}
\pscircle(0,1){0.3}
\pscircle(0,2){0.3}
\pscircle(-1,3){0.3}
\pscircle(1,3){0.3}
\pscircle(0,4){0.3}
\pscircle(0,5){0.3}
\pscircle(0,6){0.3}

\psline(0,0)(0,1)
\psline(0,2)(-1,3)
\psline(1,3)(0,4)
\psline(0,5)(0,6)

\psbezier(0,0)(-0.7,0.7)(-0.7,1.3)(0,2)
\psbezier[linestyle=dashed](0,2)(-0.7,2.7)(-0.7,3.3)(0,4)
\psbezier(0,4)(0.7,4.7)(0.7,5.3)(0,6)
\psbezier(0,1)(0.7,1.7)(1,2.3)(1,3)
\psbezier(-1,3)(-1,3.7)(-0.7,4.3)(0,5)

\end{pspicture}
\end{center}
\caption{$\SA(1)$}
\label{fig:A(1)}
\end{figure}

\begin{lemma}\mbox{}
\label{lem:subalg-mult}
\begin{enumerate}
\item
The kernel of right multiplication by $\Sq{1}$ on $\SA$
equals the image of right multiplication by $\Sq{1}$.
\item
The kernel of right multiplication by $\SQ{1}$ on $\SA//\SE(0)$
equals the image of right multiplication by $\SQ{1}$.
\item
The kernel of right multiplication by $\Sq{2}$ on $\SA//\SE(1)$
equals the image of right multiplication by $\Sq{2}$.
\end{enumerate}
\end{lemma}

\begin{proof}
All three claims follow in
the motivic case for the same combinatorial reasons as in the classical
case.
\end{proof}

\begin{definition}\label{rem:M[k,l]} 
For any $\SA(1)$-module $M$, let $\Sigma^{k,l}M$ denote the $\SA(1)$-module obtained by increasing the bidegree of each element of $M$ by $(k,l).$
Let $\Sigma$ be $\Sigma^{1,0}$.
\end{definition}

\begin{lemma}
\label{lem:subalg-exact}
There are short exact sequences
\[
\xymatrix{
0 \ar[r] & \Sigma \SA//\SE(0) \ar[r]^-{\cdot \Sq{1}} & 
  \SA \ar[r] & \SA//\SE(0) \ar[r] & 0 \\
0 \ar[r] & \Sigma^{3,1} \SA//\SE(1) \ar[r]^-{\cdot \SQ{1}} 
  & \SA//\SE(0) \ar[r] & \SA//\SE(1) \ar[r] & 0 \\
0 \ar[r] & \Sigma^{2,1} \SA//\SA(1) \ar[r]^-{\cdot \Sq{2}} & \SA//\SE(1) \ar[r] &
  \SA//\SA(1) \ar[r] & 0. }
\]
\end{lemma}

\begin{proof}
This follows immediately from Lemma \ref{lem:subalg-mult}.
\end{proof}

\begin{definition} 
For $n=2k$, let $Q_n$ be the quadric hypersurface of 
$\C\p^{n+1}$ defined by $x_0 x_1 + \cdots + x_{2k} x_{2k+1}=0$. 
For $n=2k+1$, let $Q_{n}$ be the quadric hypersurface of
$\C\p^{n+1}$ defined by 
$x_0 x_1 + \cdots + x_{2k} x_{2k+1} + x_{2k+2}^2 =0$.
Let $DQ_{n}$ be the open complement $\C\p^{n}-Q_{n-1}$ of $Q_{n-1}$ in $\C\p^n$.
\end{definition}

Consider the inclusions
\[
\C\p^{2k} \map \C\p^{2k+1}: [x_0 : \cdots : x_{2k}] \mapsto
[x_0: \cdots : x_{2k-1} : x_{2k} : x_{2k} ]
\]
and
\[
\C\p^{2k-1} \map \C\p^{2k}:
[x_0 : \cdots : x_{2k-1} ] \mapsto
[x_0 : \cdots : x_{2k-1} : 0 ].
\]
These maps restrict to inclusions $DQ_n \map DQ_{n+1}$.  We will implicitly
assume that $DQ_n$ is a subvariety of $DQ_{n+1}$.

\begin{definition}\label{def:DQinf} 
Let $DQ_{\infty}$ be $\colim_n(DQ_n)$.
\end{definition}

All cohomology groups throughout the paper are taken with 
$\F_2$-coefficients.

\begin{proposition}
\cite{DDDI}
\cite{V2}
\begin{align*}
H^{*,*}(DQ_{n}) &=
\left\{
\begin{array}{lll}
\M_2[a,b] / a^2 - \tau b, b^{k+1} & \mbox{ if } & n=2k+1 \\
\M_2[a,b] / a^2 - \tau b, b^{k+1}, ab^k  & \mbox{ if } & n=2k
\end{array}
\right. \\
H^{*,*} (DQ_\infty) & =
\M_2[a,b] / a^2 - \tau b,
\end{align*}
where $a$ has bidegree $(1,1)$ and $b$ has bidegree $(2,1)$.
\end{proposition}

\begin{proposition}
\label{prop:DQinfty-A(1)-module}
The $\SA(1)$-module structure on $H^{*,*}(DQ_n)$ is given by:
\begin{align*}
\Sq{1} ab^k & = b^{k+1} \\
\Sq{1} b^k & = 0 \\
\Sq{2} b^k & = 
\left\{
\begin{array}{ll}
0 & \mbox{ if } k \mbox{ is even} \\
b^{k+1} & \mbox{ if } k \mbox{ is odd}
\end{array}
\right. \\
\Sq{2} ab^k & = 
\left\{
\begin{array}{ll}
0 & \mbox{ if } k \mbox{ is even} \\
ab^{k+1} & \mbox{ if } k \mbox{ is odd}.
\end{array}
\right. 
\end{align*}
\end{proposition}

\begin{proof}
As proved in \cite[Lemma 4.7]{DDDI} and \cite{V2},
we have that $\Sq{1}a=b$.  Since $\Sq{2}$ has degree $(2,1)$ and $a$ has degree $(1,1)$, properties of the motivic Steenrod algebra given in \cite[Lemma 9.9]{V2} imply that $\Sq{2}a=0$. 

From the definition of $\SA(1)$, we have $\Sq{1}b=\Sq{1,1}a=0$. 
Also, results in \cite[Lemma 9.8]{V2} imply that $\Sq{2}b=b^2$. 

The $\SA(1)$-action on the other elements of $H^{*,*}(DQ_n)$ 
follows from the Cartan formula \cite[Proposition 9.6]{V2}.
\end{proof}

\begin{remark}
The $\SA(1)$-module structure on $H^{*,*}(DQ_n)$ is easily derived
from the above proposition, using the map $DQ_n \map DQ_{\infty}$.
\end{remark}

Figure \ref{fig:DQ_inf} is a pictorial representation of
the $\SA(1)$-module $H^{*,*}(DQ_\infty)$, where
a circle at height $p$ stands for a copy of $\M_2$ 
in topological degree $p$.  Multiplications
by $\Sq{1}$ are indicated by straight lines, and multiplications
by $\Sq{2}$ are indicated by curved lines.

\begin{figure}[htbp!]
\begin{center}
\psset{unit=0.5cm}
\begin{pspicture}(-2,0)(2,11)

\pscircle(0,0){0.3}
\pscircle(0,1){0.3}
\pscircle(0,2){0.3}
\pscircle(0,3){0.3}
\pscircle(0,4){0.3}
\pscircle(0,5){0.3}
\pscircle(0,6){0.3}
\pscircle(0,7){0.3}
\pscircle(0,8){0.3}
\pscircle(0,9){0.3}
\pscircle(0,10){0.3}

\psline(0,1)(0,2)
\psline(0,3)(0,4)
\psline(0,5)(0,6)
\psline(0,7)(0,8)
\psline(0,9)(0,10)

\psbezier(0,2)(0.7,2.7)(0.7,3.3)(0,4)
\psbezier(0,3)(-0.7,3.7)(-0.7,4.3)(0,5)
\psbezier(0,6)(0.7,6.7)(0.7,7.3)(0,8)
\psbezier(0,7)(-0.7,7.7)(-0.7,8.3)(0,9)

\rput(-2,0){$1$}
\rput(-2,1){$a$}
\rput(-2,2){$b$}
\rput(-2,3){$ab$}
\rput(-2,4){$b^2$}
\rput(-2,5){$ab^2$}
\rput(-2,6){$b^3$}
\rput(-2,7){$ab^3$}
\rput(-2,8){$b^4$}
\rput(-2,9){$ab^4$}
\rput(-2,10){$b^5$}

\rput(2,0){$(0,0)$}
\rput(2,1){$(1,1)$}
\rput(2,2){$(2,1)$}
\rput(2,3){$(3,2)$}
\rput(2,4){$(4,2)$}
\rput(2,5){$(5,3)$}
\rput(2,6){$(6,3)$}
\rput(2,7){$(7,4)$}
\rput(2,8){$(8,4)$}
\rput(2,9){$(9,5)$}
\rput(2,10){$(10,5)$}

\rput(0,11){$\vdots$}

\end{pspicture}
\end{center}
\caption{Action of $\SA(1)$ on $H^{*,*}(DQ_\infty)$}\label{fig:DQ_inf}
\end{figure}
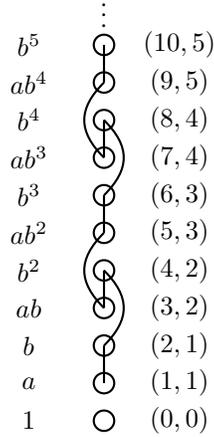

\section{Background on motivic spectra}

In this paper, we work with stable motivic homotopy theory
over $\Spec \C$, after completion at the Eilenberg-Mac Lane spectrum
$H\F_2$ in the sense of \cite{DI3}.  We suppress the completion from
the notation.  (For example, $\pi_{p,q}$ is the $H\F_2$-completed
motivic stable homotopy group.)
The results of
\cite{HKO} imply that $H\F_2$-completion is the same as $2$-completion.

We refer to \cite[Part 3]{MHT} for background on motivic stable
homotopy theory.  There are several well-behaved model structures
for motivic stable homotopy theory.  Although we will use
model theoretic techniques occasionally, we will not be precise
about these technical details.  For example, we shall
implicitly assume that all motivic spectra are cofibrant
and fibrant.

If $X$ is a based motivic space, then we abuse notation and write
$X$ also for the motivic suspension spectrum of $X$.

The following calculation is essential to our results.

\begin{proposition}\mbox{}
\label{prop:pi-compute}
\begin{enumerate}
\item
The subring $\pi_{0,*} = \oplus_k \pi_{0,k}$ 
of the motivic stable homotopy ring is equal to
$\Z_2[\tau]$, where $\tau$ has bidegree $(0,-1)$.
\item
The motivic stable homotopy group $\pi_{p,q}$ is zero if $p<0$ or $q>p$.
\end{enumerate}
\end{proposition}

\begin{proof}
This follows from motivic versions of 
the Adams spectral sequence \cite{DI3} or
the Adams-Novikov spectral sequence \cite{HKO}.
\end{proof}

For any $X$ and any $j$, observe that 
$\pi_{j,*} X = \oplus_k \pi_{j,k} X$ is a module over $\Z_2[\tau]$.

We recall the notion of cellular spectra from \cite[Definition 2.10]{DDDI2}.

\begin{definition}
\label{defn:cellular}
The class of cellular motivic spectra is the smallest class such that:
\begin{enumerate}
\item
Every sphere $S^{p,q}$ is cellular;
\item
If $X$ is weakly equivalent to a cellular motivic spectrum, then it is cellular;
\item
If $X$ is a homotopy colimit of cellular motivic spectra, then it is cellular.
\end{enumerate}
\end{definition}

We will need the following Whitehead theorem for
cellular motivic spectra.

\begin{proposition}[DDDI2, Corollary 7.2]
\label{prop:pi-cellular}
A map between cellular motivic spectra is a weak equivalence if 
and only if it induces an isomorphism on $\pi_{p,q}$ for all $p$ and $q$.
\end{proposition}

Recall that the motivic Eilenberg-Mac Lane spectrum $H\F_2$ 
represents motivic cohomology with $\F_2$-coefficients.

\begin{proposition}[\cite{HKO}]
The motivic spectrum $H\F_2$ is cellular.
\end{proposition}

The following proposition is proved in 
\cite{DI-flasque} and \cite{Jardine}.

\begin{proposition}
The functor $\pi_{p,q}$ commutes with filtered colimits of motivic spectra.
\end{proposition}

Finally, we recall that there is a topological realization functor
from motivic spectra to ordinary spectra \cite{DDDI3} \cite[Section 3.3]{MV}.
It is built from the functor
that takes a complex algebraic variety to its space of $\C$-valued
points.  

Topological realization preserves homotopy colimits, and in particular
preserves cofiber sequences.  As a result, all of our motivic calculations
must be compatible with analogous classical calculations.  For example,
there are maps of long exact sequences and spectral sequences 
from the motivic setting to 
the classical setting, but we will not make this precise.

\subsection{Equivariant motivic spectra}

Let $G$ be a group.  (In our application, we are concerned only with the
group $\Z/2$.)
For general reasons, there is a convenient homotopy theory
of $G$-equivariant motivic spectra, where weak equivalences are detected
by the underlying non-equivariant motivic spectra.
The homotopy theory of $G$-equivariant motivic spectra has an internal
function object.  This means that for any equivariant motivic spectra
$X$ and $Y$, there is an equivariant motivic spectrum
$F_G(X,Y)$ with appropriate adjointness properties.
There are other possible homotopy theories for
equivariant motivic spectra, but this is the one that is relevant to
homotopy fixed point constructions.

We apply the standard notion 
of a homotopy fixed point spectrum to our motivic setting.
Among many other places in the literature, homotopy fixed point spectra
are treated in \cite[Section 2.1]{Ostvaer}.  

The following definition is lifted straight from the classical setting.

\begin{definition}
Let $X$ be a motivic spectrum with an action by a group $G$. 
Define the homotopy fixed point spectrum $X^{hG}$ of $X$ 
to be $F_{G}(EG_{+},X).$
\end{definition}

\begin{lemma}
If $X$ is a ring spectrum with an action by $G$, then
$X^{hG}$ is also a ring spectrum.
\end{lemma}

\begin{proof}
The multiplication map
$X^{hG} \Smash X^{hG} \map X^{hG}$ is the composition
\[
F_G(EG_+,X) \Smash F_G(EG_+,X) \map 
F_G(EG_+ \Smash EG_+, X \Smash X) \map
F_G(EG_+, X),
\]
where the first map is smash product of maps, and the second map
is induced by the diagonal $EG_+ \map EG_+ \Smash EG_+$ and the
multiplication $X \Smash X \map X$.
\end{proof}

The following theorem can be proved in the same way as the
analogous classical result.

\begin{theorem}
\label{motivicHFP} 
Let $G$ be a finite group, and let $X$ be a $G$-equivariant motivic spectrum.
The motivic homotopy fixed point spectral sequence is conditionally convergent
and takes the form
\[
E_{2}^{n,p,u}=H^{p}(G;\pi_{n+p,u}X) \Rightarrow
\pi_{n,u}(X^{hG}),
\]
where $H^p(G;-)$ is group cohomology.  
\end{theorem}

\subsection{Motivic connective covers}
\label{sec:ConCov}
We continue to work in $H\F_2$-complete stable motivic homotopy
theory over $\Spec \C$.

\begin{proposition}
\label{prop:factor}
For any map $f:X \map Y$ of motivic spectra, there is a factorization
\[
\xymatrix{
X \ar[r]^i & Z \ar[r]^p & Y
}
\]
of $f$ such that:
\begin{enumerate}
\item
$p$ is an isomorphism on $\pi_{a,b}$ if $a \geq 1$ and $a-b \geq 1$.
\item
$i$ is an isomorphism on $\pi_{a,b}$ if $a < 0$ or $a-b < 0$.
\item
$p$ is injective on $\pi_{0,b}$ if $b \leq 0$
and on $\pi_{a,a}$ if $a \geq 0$.
\item
$i$ is surjective on $\pi_{0,b}$ if $b \leq 0$ and
on $\pi_{a,a}$ if $a \geq 0$.
\end{enumerate}
\end{proposition}

\begin{proof}
It is easiest to work with model categories.  We may assume that
$X$ and $Y$ are cofibrant and fibrant.

Let $S^{p,q} \map D^{p,q}$ be a cofibration whose target is contractible.
Construct
$Z$ by applying the small object argument to the 
set of generating acyclic cofibrations together with all maps of the form
$S^{p,q} \map D^{p,q}$, with $p \geq 0$ and $p-q \geq 0$.
Proposition \ref{prop:pi-compute}(2) ensures that $Z$ has the desired
motivic homotopy groups.
\end{proof}

\begin{remark}
The factorizations produced in Proposition \ref{prop:factor}
are functorial in $f$.
\end{remark}

\begin{corollary}
\label{cor:Postnikov}
Given a motivic spectrum $X$, there exists a motivic spectrum $PX$ and a map 
$X \rightarrow PX$ such that:
\begin{enumerate}
\item $\pi_{p,q} X \rightarrow \pi_{p,q} PX$ is an isomorphism for $p < 0$ 
or $p-q <0$, and
\item $\pi_{p,q} PX = 0$ if $p \geq 0$ and $p-q \geq 0$.
\end{enumerate}
\end{corollary}

\begin{proof}
Apply Proposition \ref{prop:factor} to the map $X \map *$.
\end{proof}

\begin{corollary}
\label{cor:ConCov}
Given a motivic spectrum $X$, there exists a motivic spectrum $CX$ and a map 
$CX \rightarrow X$ such that:
\begin{enumerate}
\item $\pi_{p,q} CX \rightarrow \pi_{p,q} X$ is an isomorphism for $p \geq 0$ 
and $p-q \geq 0$, and
\item $\pi_{p,q} CX = 0$ if $p < 0$ or $p-q < 0$.
\end{enumerate}
\end{corollary}

\begin{proof}
Apply Proposition \ref{prop:factor} to the map $* \map X$.
\end{proof}

\begin{definition}
\label{defn:ConCov}
The motivic spectrum $CX$ of Corollary \ref{cor:ConCov} is the
connective cover of $X$.
\end{definition}

\begin{remark} 
\label{rem:action} 
Let $G$ be a group.
The proof of Proposition \ref{prop:factor} works just as well in the
category of $G$-equivariant motivic spectra.  Therefore,
we can construct connective covers of 
$G$-equivariant motivic spectra.
\end{remark}

\subsection{Positive cellular motivic spectra}

\begin{definition}
\label{defn:positive}
The class of positive cellular motivic spectra is the smallest class of 
motivic spectra such that:
\begin{enumerate}
\item
$S^{p,q}$ is positive cellular if $p \geq 0$ and $p-q \geq 0$.
\item
If $f:X \map Y$ is any map such that $X$ and $Y$ are both positive cellular, 
then the cofiber of $f$ is also positive cellular.
\item
If $X$ is weakly equivalent to a positive cellular motivic spectrum,
then $X$ is positive cellular.
\item
The filtered colimit of a diagram of positive cellular motivic spectra
is positive cellular.
\end{enumerate}
\end{definition}

In essence, a positive cellular motivic spectrum is one that can be built
by attaching cells along the spheres $S^{p,q}$ with $p \geq 0$ and $p-q \geq 0$.

\begin{proposition}
The class of positive cellular motivic spectra is equal to the class
of cellular motivic spectra $X$ such that 
$\pi_{a,b} X$ is zero
if $a < 0$ or $a-b<0$.  
\end{proposition}

\begin{proof}
Consider the class of motivic spectra $X$ such that
$\pi_{a,b} X$ is zero if $a<0$ or $a-b<0$.  This class satisfies
the four properties of Definition \ref{defn:positive}.
Also, the class of cellular motivic spectra satisfies these four
properties.  This shows that
if $X$ is a positive cellular motivic spectrum, then 
$X$ is cellular and $\pi_{a,b} X$ is zero if $a<0$ or $a-b<0$.

Now suppose that $X$ is a cellular motivic spectrum such that
$\pi_{a,b} X$ is zero if $a<0$ or $a-b<0$.
Apply Proposition \ref{prop:factor} to the map $* \map X$ to obtain a map
$f:Z \map X$.  By construction, $Z$ is a positive cellular motivic spectrum,
and Proposition \ref{prop:factor} guarantees that 
$f$ induces an isomorphism on $\pi_{a,b}$ for all $a$ and $b$.
Since $Z$ and $X$ are both cellular, it follows from 
Proposition \ref{prop:pi-cellular} 
that $f$ is a weak equivalence.  Hence $X$ is positive.
\end{proof}

\begin{lemma}
\label{lem:positive-suspend}
Let $p \geq 0$ and $p-q \geq 0$.
If $X$ is a positive cellular motivic spectrum, then so is
$\Sigma^{p,q} X$.
\end{lemma}

\begin{proof}
Consider the class of all motivic spectra such that
$\Sigma^{p,q} X$ is positive cellular.  We would like to show
that this class contains the positive cellular motivic spectra,
so it is enough to show that the class satisfies the four properties
of Definition \ref{defn:positive}.  These properties are easy to check.
\end{proof}

\begin{proposition}
\label{prop:positive-smash}
If $X$ and $Y$ are both positive cellular motivic spectra, then so is
$X \Smash Y$.
\end{proposition}

\begin{proof}
Fix a positive cellular motivic spectrum $X$, and consider the
class of all motivic spectra $Z$ such that $X \Smash Z$
is a positive cellular motivic spectrum.  We would like to show
that this class contains all positive motivic cellular spectra.
Thus, we only need to show that the class satisfies the four properties
of Definition \ref{defn:positive}.

Property (1) is Lemma \ref{lem:positive-suspend}.
Property (2) follows from the fact that smash product with $X$
preserves cofiber sequences.
Property (3) follows from the fact that smash product with $X$
preserves weak equivalences.
Property (4) follows from the fact that smash product with $X$
commutes with filtered colimits.
\end{proof}

\begin{lemma}
\label{lem:pos-neg}
Let $X$ be a positive cellular motivic spectrum, and let
$Y$ be a motivic spectrum such that $\pi_{a,b} Y$ is zero
if $a \geq 0$ and $a-b \geq 0$.
Then $[X,Y]$ equals zero.
\end{lemma}

\begin{proof}
Recall that the category of motivic spectra is enriched over simplicial
sets.  This means that for all motivic spectra $X$ and $Y$, there
is a simplicial set $\Map(X,Y)$ such that
$[\Sigma^n X,Y]$ equals $\pi_n \Map(X,Y)$.

Consider the class of motivic spectra $W$ such that
$\Map(W,Y)$ is contractible.
It suffices to show that this class satisfies the four
properties of Definition \ref{defn:positive}.

Property (1) is satisfied by the assumption on $Y$.
Property (2) is satisfied by the fact that the functor
$\Map(-,Y)$ takes cofiber sequences of motivic spectra
to fiber sequences of simplicial sets.
Property (3) follows from the fact that $\Map(-,Y)$
takes weak equivalences of motivic spectra to weak equivalences of
simplicial sets.
Property (4) follows from the fact that $\Map(-,Y)$
takes filtered colimits, which are homotopy colimits,
to homotopy limits.
\end{proof}

\begin{corollary}
\label{cor:pos-neg-isomorphism}
Let $X$, $Y$, and $Z$ be cellular motivic spectra.  Suppose that
$X$ is positive, and suppose that there is a map $Y \map Z$
that is:
\begin{enumerate}
\item
injective on $\pi_{a,b}$ if $a \geq 0$ and $a-b \geq 0$;
\item
surjective on $\pi_{a,b}$ if $a \geq -1$ and $a-b \geq -1$.
\end{enumerate}
Then
$[X,Y] \map [X,Z]$
is an isomorphism.
\end{corollary}

\begin{proof}
Let $F$ be the fiber of $Y \map Z$.  
By the long exact sequence
\[
\cdots \map \pi_{a+1,b} Y \map \pi_{a+1,b} Z \map
\pi_{a,b} F \map \pi_{a,b} Y \map \pi_{a,b} Z \map \cdots,
\]
$\pi_{a,b} F$ is zero if $a \geq 0$ and $a-b \geq 0$.
It follows from Lemma \ref{lem:pos-neg} that
$\Map(X,F)$ is contractible.
By consideration of the fiber sequence
\[
\Map(X,F) \map \Map(X,Y) \map \Map(X,Z)
\]
of simplicial sets, it follows that
$\Map(X,Y) \map \Map(X,Z)$ is a weak equivalence.
\end{proof}

\begin{corollary}
\label{cor:pos-neg-isomorphism-2}
Let $X$, $Y$, and $Z$ be cellular motivic spectra.  Suppose that
$\pi_{a,b} X$ is zero if $a \geq 0$ and $a-b \geq 0$.
Also suppose that there is a map $Y \map Z$
that is:
\begin{enumerate}
\item
injective on $\pi_{a,b}$ if $a < 1$ or $a-b < 1$;
\item
surjective on $\pi_{a,b}$ if $a < 0$ or $a-b < 0$.
\end{enumerate}
Then
$[Y,X] \map [Z,X]$
is an isomorphism.
\end{corollary}

\begin{proof}
Let $C$ be the cofiber of $Y \map Z$.
By the long exact sequence
\[
\cdots \map \pi_{a,b} Y \map \pi_{a,b} Z \map
\pi_{a,b} C \map \pi_{a-1,b} Y \map \pi_{a-1,b} Z \map \cdots,
\]
$\pi_{a,b} C$ is zero if $a < 0$ or $a-b < 0$.
In other words, $C$ is positive cellular.
It follows from Lemma \ref{lem:pos-neg} that
$\Map(C,X)$ is contractible.
By consideration of the fiber sequence
\[
\Map(C,X) \map \Map(Y,X) \map \Map(Z,X)
\]
of simplicial sets, it follows that
$\Map(Y,X) \map \Map(Z,X)$ is a weak equivalence.
\end{proof}

The following result is a straightforward motivic version of 
\cite[Lemma 2.11]{May}.

\begin{proposition}
\label{prop:ConCov-ring}
If $X$ is a cellular motivic ring spectrum, 
then $CX$ has a unique (up to homotopy)
multiplication such that the map
$CX \map X$ is a map of motivic ring spectra.
\end{proposition}

\begin{proof}
First note that $CX$ is a positive cellular motivic spectrum.
By Proposition \ref{prop:positive-smash},
$CX \Smash CX$ is a positive cellular motivic spectrum,
so the map
\[
[CX \Smash CX, CX] \map [CX \Smash CX, X]
\]
is an isomorphism
by Corollary \ref{cor:pos-neg-isomorphism}.
Therefore, there is a unique homotopy class of maps $CX \Smash CX \map CX$
such that the diagram
\[
\xymatrix{
CX \Smash CX \ar[r] \ar@{-->}[d] & X \Smash X \ar[d] \\
CX \ar[r] & X  }
\]
commutes.
\end{proof}

Recall that classical Eilenberg-Mac Lane spectra are unique in the
following sense.  If $H$ and $H'$ are any two spectra such that
$\pi_0 H$ and $\pi_0 H'$ are isomorphic while $\pi_k H$ and
$\pi_k H'$ are zero if $k \neq 0$, then $H$ and $H'$ are
weakly equivalent.  We next prove a motivic version.

\begin{proposition}
\label{prop:EM-unique}
Suppose that $H$ and $H'$ are cellular motivic spectra such that
$\pi_{0,*} H$ and $\pi_{0,*} H'$ are isomorphic as $\Z_2[\tau]$-modules.
Suppose also that $\pi_{a,b} H$ and $\pi_{a,b} H'$
are zero unless $a=0$ and $b \leq 0$.  Then $H$ and $H'$
are weakly equivalent.
\end{proposition}

\begin{proof}
Let $N$ be the $\Z_2[\tau]$-module 
$\pi_{0,*} H = \pi_{0,*} H'$.
Consider the motivic Moore spectrum $C$
constructed by the cofiber sequence
\[
\vee_\alpha S^{0,k_\alpha} \map \vee_\beta S^{0,j_\beta} \map C,
\]
where the second wedge is indexed by a set of generators of $N$ and
the first wedge is indexed by a set of defining relations for $N$.
A straightforward calculation shows that $\pi_{0,*} C$ is isomorphic
to $N$ and $\pi_{a,b} C$ is zero if $a<0$.
Since $N$ is concentrated in degrees $(0,k)$ with $k\leq 0$,
it follows that $\pi_{a,b} C$ is zero if $a-b < 0$.
There are obvious maps $C \map H$ and $C \map H'$ 
that induce isomorphisms on $\pi_{0,*}$.

Let $F$ be the fiber of $C \map H$.  
From the long exact sequence in homotopy groups, it follows
that $\pi_{a,b} F$ is zero if $a<1$ or $a-b<0$.
Thus $\pi_{a,b} (\Sigma^{-1,-1} F)$ is zero
if $a<0$ or $a-b < 0$.

Also, $\pi_{a,b} (\Sigma^{-1,-1} H')$ is zero if $a \geq 0$
and $a-b \geq 0$.  In other words, $\Sigma^{-1,-1} H'$ is positive cellular.
By Lemma \ref{lem:pos-neg}, $[\Sigma^{-1,-1} F, \Sigma^{-1,-1} H']$
and hence $[F, H']$ are both zero.
It follows that the map $[H,H'] \map [C,H']$ is surjective.
In particular, there exists a map $H \map H'$ making the diagram
\[
\xymatrix{C \ar[r]\ar[dr] & H \ar[d] \\ & H'}
\]
commute.  This map $H \map H'$
is an isomorphism on $\pi_{*,*}$, so it is a weak equivalence
by Proposition \ref{prop:pi-cellular}.
\end{proof}

\begin{remark}
Proposition \ref{prop:EM-unique} generalizes in a straightforward way
to cellular motivic spectra $H$ such that $\pi_{0,*} H$ is bounded above
in the sense that $\pi_{0,k} H$ is zero for $k$ greater than some fixed $n$.
\end{remark}

\section{Motivic $K$-theory spectra}

We remind the reader that we are working with $H\F_2$-complete
stable motivic homotopy theory over $\Spec \C$.

\subsection{$KGL$}

Recall that $KGL$ is the algebraic $K$-theory spectrum 
\cite[Part 3, Section 3.2]{MHT}.  This is a ring spectrum; in fact,
it has a strictly associative and commutative multiplication \cite{RSO}.
In \cite[Theorem 6.2]{DDDI2} it is shown that $KGL$ is cellular. 

The group schemes $GL_n$ are equipped with an involution given by
inverse-transpose.  These involutions extend to an involution on $GL$,
and then to $BGL$.  In this way, the motivic spectrum $KGL$ is
a $\Z/2$-equivariant motivic spectrum.

\begin{proposition}
\label{prop:pi-KGL}
The ring $\pi_{*,*} KGL$ is equal to $\Z_2 [\tau, \beta^{\pm 1}]$, where
$\tau$ has bidegree $(0,-1)$ and $\beta$ has bidegree $(2,1)$.
The involution on $\pi_{*,*} KGL$ fixes $\tau$ but takes $\beta$ to $-\beta$.
\end{proposition}

\begin{proof}
The first part follows from the calculation of
$2$-complete algebraic $K$-theory of $\C$ \cite{Suslin}, together
with the fact that $\pi_{p,q} KGL$ is isomorphic
to $K_{p-2q} (\C)$ \cite[Section 4, Theorem 3.13]{MV}.

The second part follows from the fact that the involution on $KGL$
takes the line bundle $\mathcal{O}(-1)$ on $\p^1$ to $\mathcal{O}(1)$.
\end{proof}

\subsection{$KO$}

\begin{definition}
Let $KO$ be the homotopy fixed points spectrum $KGL^{h\Z/2}$.
\end{definition}

\begin{remark}
Recall that in the classical case, the real $K$-theory spectrum
$KO$ is weakly equivalent to $KU^{h\Z/2}$.  Our terminology is
chosen to emphasize this analogy.  
\end{remark}

\begin{remark}
We warn the reader that we are
not claiming that $KO$ represents Hermitian $K$-theory,
although we suspect that this is true.
\end{remark}

Our next goal is to compute the homotopy of $KO$.  We will
start with the homotopy of $KGL$ and apply
the homotopy fixed points spectral sequence.

\begin{proposition} 
\label{prop:KGL2} 
The $E_{2}$-page of the homotopy fixed points spectral sequence for
$KO$  is
$\Z_{2}[\tau,h_{1},c^{\pm 1}]/2h_{1}$,
where the degree of $\tau$ is $(0,0,-1)$, 
the degree of $h_{1}$ is $(1,1,1)$,
and the degree of $c$ is $(4,0,2)$.
\end{proposition}

\begin{proof}
This is a straightforward calculation, using the classical computations
of $H^*(\Z/2; \Z)$ and $H^*(\Z/2;\Z(-1))$,
where $\Z(-1)$ is the $\Z/2$-module whose involution is
multiplication by $-1$.
\end{proof}

Figure \ref{fig:E^{2}} is a pictorial representation of the
computation in Proposition \ref{prop:KGL2}.
Here, and in the figures following it, $E_2^{n,p,u}$ is located at
coordinates $(n,p)$, and the weight is not shown.
Copies of $\Z_{2}[\tau]$ are represented by open boxes.
Copies of $\F_2[\tau]$ are represented by solid circles.
Copies of $\F_2[\tau]/\tau$ are represented by open circles.
Lines of slope $1$ represent multiplications by $h_1$.

\begin{center}
\begin{figure}[htbp!]
\psset{unit=0.5cm}
\begin{pspicture}(-10,-2)(10,10)
\scriptsize

\psline(-10,-1)(10,-1)
\psline(-10,-1)(-10,9)

\rput(0,-1.6){$0$}
\rput(-4,-1.6){$-4$}
\rput(-8,-1.6){$-8$}
\rput(4,-1.6){$4$}
\rput(8,-1.6){$8$}

\psframe(-0.25,-0.25)(0.25,0.25)
\psline(0,0)(1,1)
\pscircle*(1,1){0.2}
\psline(1,1)(2,2)
\pscircle*(2,2){0.2}
\psline(2,2)(3,3)
\pscircle*(3,3){0.2}
\psline(3,3)(4,4)
\pscircle*(4,4){0.2}
\psline(4,4)(5,5)
\pscircle*(5,5){0.2}
\psline(5,5)(6,6)
\pscircle*(6,6){0.2}
\psline(6,6)(7,7)
\pscircle*(7,7){0.2}
\psline(7,7)(8,8)
\pscircle*(8,8){0.2}
\psline(8,8)(9,9)


\psframe(-4.25,-0.25)(-3.75,0.25)
\psline(-4,0)(-3,1)
\pscircle*(-3,1){0.2}
\psline(-3,1)(-2,2)
\pscircle*(-2,2){0.2}
\psline(-2,2)(-1,3)
\pscircle*(-1,3){0.2}
\psline(-1,3)(0,4)
\pscircle*(0,4){0.2}
\psline(0,4)(1,5)
\pscircle*(1,5){0.2}
\psline(1,5)(2,6)
\pscircle*(2,6){0.2}
\psline(2,6)(3,7)
\pscircle*(3,7){0.2}
\psline(3,7)(4,8)
\pscircle*(4,8){0.2}
\psline(4,8)(5,9)


\psframe(-8.25,-0.25)(-7.75,0.25)
\psline(-8,0)(-7,1)
\pscircle*(-7,1){0.2}
\psline(-7,1)(-6,2)
\pscircle*(-6,2){0.2}
\psline(-6,2)(-5,3)
\pscircle*(-5,3){0.2}
\psline(-5,3)(-4,4)
\pscircle*(-4,4){0.2}
\psline(-4,4)(-3,5)
\pscircle*(-3,5){0.2}
\psline(-3,5)(-2,6)
\pscircle*(-2,6){0.2}
\psline(-2,6)(-1,7)
\pscircle*(-1,7){0.2}
\psline(-1,7)(0,8)
\pscircle*(0,8){0.2}
\psline(0,8)(1,9)


\psline(-9,3)(-8,4)
\pscircle*(-8,4){0.2}
\psline(-8,4)(-7,5)
\pscircle*(-7,5){0.2}
\psline(-7,5)(-6,6)
\pscircle*(-6,6){0.2}
\psline(-6,6)(-5,7)
\pscircle*(-5,7){0.2}
\psline(-5,7)(-4,8)
\pscircle*(-4,8){0.2}
\psline(-4,8)(-3,9)


\psline(-9,7)(-8,8)
\pscircle*(-8,8){0.2}
\psline(-8,8)(-7,9)


\psframe(3.75,-0.25)(4.25,0.25)
\psline(4,0)(5,1)
\pscircle*(5,1){0.2}
\psline(5,1)(6,2)
\pscircle*(6,2){0.2}
\psline(6,2)(7,3)
\pscircle*(7,3){0.2}
\psline(7,3)(8,4)
\pscircle*(8,4){0.2}
\psline(8,4)(9,5)


\psframe(7.75,-0.25)(8.25,0.25)
\psline(8,0)(9,1)

\psline{->}(4,0)(3.1,2.7)
\psline{->}(5,1)(4.1,3.7)
\psline{->}(6,2)(5.1,4.7)
\psline{->}(7,3)(6.1,5.7)
\psline{->}(8,4)(7.1,6.7)
\psline{->}(9,5)(8.1,7.7)

\psline{->}(-4,0)(-4.9,2.7)
\psline{->}(-3,1)(-3.9,3.7)
\psline{->}(-2,2)(-2.9,4.7)
\psline{->}(-1,3)(-1.9,5.7)
\psline{->}(0,4)(-0.9,6.7)
\psline{->}(1,5)(0.1,7.7)
\psline{->}(2,6)(1.1,8.7)
\psline(3,7)(2.33,9)
\psline(4,8)(3.67,9)

\psline{->}(-8,4)(-8.9,6.7)
\psline{->}(-7,5)(-7.9,7.7)
\psline{->}(-6,6)(-6.9,8.7)
\psline(-5,7)(-5.67,9)
\psline(-4,8)(-4.33,9)

\end{pspicture}

\caption{The $E_{2}$-page of the homotopy fixed points spectral sequence
for $KO$}\label{fig:E^{2}}
\end{figure}
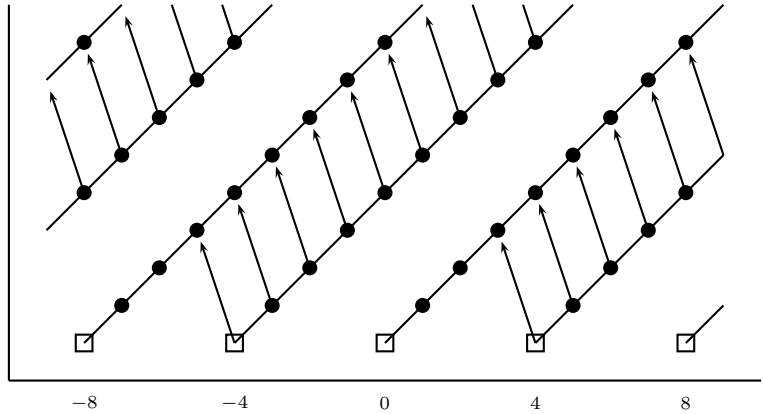
\end{center}

For dimension reasons, the $d_2$-differential is zero.

\begin{lemma}
\label{lem:d3}
In the homotopy fixed points spectral sequence for $KO$,
$d_3(\tau) = 0$, $d_3(h_1) = 0$, and $d_3(c) = \tau h_1^3$.
\end{lemma}

\begin{proof}
This follows immediately by degree considerations and topological
realization.
Classically, $d_3$ takes $h_1$ to $0$ and takes $c$ to $h_1^3$.
\end{proof}

\begin{proposition}
\label{prop:E^{4}}
The $E_\infty$-page of the homotopy fixed points spectral sequence for
$KO$ is
\[
\frac{\Z_2[\tau,h_1,a,b^{\pm 1}]}{2h_1, \tau h_1^3, a^2=4b, h_1 a},
\]
where the degree of $\tau$ is $(0,0,-1)$, 
the degree of $h_1$ is $(1,1,1)$, 
the degree of $a$ is $(4,0,2)$,
and the degree of $b$ is $(8,0,4)$.
\end{proposition}

\begin{proof}
The $E_4$-page can be computed from the description of $d_3$
given in Lemma \ref{lem:d3}.  The element $a$ corresponds to $2c$,
while $b$ corresponds to $c^2$.

Then observe that all higher
differentials vanish for dimension reasons.
\end{proof}

Figure \ref{fig:E^{4}} is a pictorial representation of the
$E_\infty$-page of the spectral sequence.  The notation is 
the same as in Figure \ref{fig:E^{2}}.

\begin{center}
\begin{figure}[htbp!]
\psset{unit=0.5cm}
\begin{pspicture}(-10,-2)(10,10)
\scriptsize

\psline(-10,-1)(10,-1)
\psline(-10,-1)(-10,9)

\rput(0,-1.6){$0$}
\rput(-4,-1.6){$-4$}
\rput(-8,-1.6){$-8$}
\rput(4,-1.6){$4$}
\rput(8,-1.6){$8$}

\psframe(-0.25,-0.25)(0.25,0.25)
\psline(0,0)(1,1)
\pscircle*(1,1){0.2}
\psline(1,1)(2,2)
\pscircle*(2,2){0.2}
\psline(2,2)(3,3)
\pscircle(3,3){0.2}
\psline(3,3)(4,4)
\pscircle(4,4){0.2}
\psline(4,4)(5,5)
\pscircle(5,5){0.2}
\psline(5,5)(6,6)
\pscircle(6,6){0.2}
\psline(6,6)(7,7)
\pscircle(7,7){0.2}
\psline(7,7)(8,8)
\pscircle(8,8){0.2}
\psline(8,8)(9,9)


\psframe(-4.25,-0.25)(-3.75,0.25)


\psframe(-8.25,-0.25)(-7.75,0.25)
\psline(-8,0)(-7,1)
\pscircle*(-7,1){0.2}
\psline(-7,1)(-6,2)
\pscircle*(-6,2){0.2}
\psline(-6,2)(-5,3)
\pscircle(-5,3){0.2}
\psline(-5,3)(-4,4)
\pscircle(-4,4){0.2}
\psline(-4,4)(-3,5)
\pscircle(-3,5){0.2}
\psline(-3,5)(-2,6)
\pscircle(-2,6){0.2}
\psline(-2,6)(-1,7)
\pscircle(-1,7){0.2}
\psline(-1,7)(0,8)
\pscircle(0,8){0.2}
\psline(0,8)(1,9)


\psline(-9,7)(-8,8)
\pscircle(-8,8){0.2}
\psline(-8,8)(-7,9)


\psframe(3.75,-0.25)(4.25,0.25)


\psframe(7.75,-0.25)(8.25,0.25)
\psline(8,0)(9,1)

\end{pspicture}

\caption{The $E_\infty$-page of the homotopy fixed points spectral sequence
for $KO$}\label{fig:E^{4}}
\end{figure}
\end{center}

\begin{theorem}
\label{theo:KGL^{hC_{2}}} 
The ring $\pi_{*,*}(KO)$ is 
\[
\frac{\Z_{2}[\tau,h_{1},a,b^{\pm 1}]}{2h_1, \tau h_{1}^{3}, a^{2}=4b, h_{1}a},
\]
where the degree of $\tau$ is $(0,-1)$, 
the degree of $h_1$ is $(1,1)$, the degree of
$a$ is $(4,2)$, 
and the degree of $b$ is $(8,4)$. 
\end{theorem}

\begin{proof} 
For dimension reasons, there are no extensions to resolve in 
the $E_{\infty}$-page described in Proposition \ref{prop:E^{4}}.
\end{proof}

\subsection{$ko$}
\label{sec:C(KGL^{hC_2})}

\begin{definition}
\label{defn:ko}
Let $ko$ be the connective cover of $KO$
in the sense of Definition \ref{defn:ConCov}.
\end{definition}

\begin{theorem}
\label{thm:main2} 
The ring $\pi_{*,*}(ko)$ is 
\[
\frac{\Z_{2}[\tau, h_{1}, a, b]}{2h_{1}, \tau h_{1}^{3}, h_{1}a, a^{2}=4b},
\]
where the degree of $\tau$ is $(0,-1)$,
the degree of $h_1$ is $(1,1)$,
the degree of $a$ is $(4,2)$,
and the degree of $b$ is $(8,4)$.
\end{theorem}

\begin{proof} 
This follows from Theorems \ref{theo:KGL^{hC_{2}}} and 
Corollary \ref{cor:ConCov}.
\end{proof}

\subsection{$kgl$ and $kgl^{h\Z/2}$}

\begin{definition}
Let $kgl$ be the connective cover of $KGL$, in the sense of 
Definition \ref{defn:ConCov}.
\end{definition}

It follows from Proposition \ref{prop:ConCov-ring} that 
$kgl$ is a ring spectrum, and the homotopy of $kgl$
is easily described from Proposition \ref{prop:pi-KGL}
and Corollary \ref{cor:ConCov}.

\begin{proposition} 
\label{prop:pi-kgl}
The ring $\pi_{*,*}(kgl)$ is isomorphic to $\Z_{2}[\tau,\beta]$, where the degree of $\tau$ is $(0,-1)$ and the degree of $\beta$ is $(2,1)$.
\end{proposition}

Since we may construct $kgl$ from $KGL$ equivariantly as in 
Remark \ref{rem:action},
it follows that $kgl$ has a $\Z/2$-action.
Next we study the homotopy fixed points spectrum
$kgl^{h\Z/2}$.

\begin{theorem} 
The ring $\pi_{*,*}(kgl^{h\Z/2})$ is 
\[
\frac{\Z_{2}[\tau, h_{1}, a, b, x]}{2h_{1}, \tau h_{1}^{3}, \tau h_{1}x, h_{1}a,
ax, a^{2}=4b, bx=h_{1}^{4}},
\]
where the degree of $\tau$ is $(0,-1)$,
the degree of $h_1$ is $(1,1)$,
the degree of $a$ is $(4,2)$,
the degree of $b$ is $(8,4)$, and
the degree of $x$ is $(-4,0)$.
\end{theorem}

\begin{proof}  
We use the homotopy fixed points spectral sequence
\[
E_2^{n,p,u} = H^p(\Z/2; \pi_{n+p,u} kgl ) \Rightarrow
\pi_{n,u} (kgl^{h\Z/2}).
\]
By direct computation of group cohomology, we find that
\[
E_{2} = \frac{\Z_{2}[\tau, h_{1}, c, z]}{2h_{1}, cz=h_{1}^{2}},
\]
where the degree of $\tau$ is $(0,0,-1)$, 
the degree of $h_1$ is $(1,1,1)$, 
the degree of $c$ is $(4,0,2)$, and
the degree of $z$ is $(-2,2,0)$.
A pictorial representation of $E_2$ is shown in 
Figure \ref{fig:E^{2}kgl}.

\begin{center}
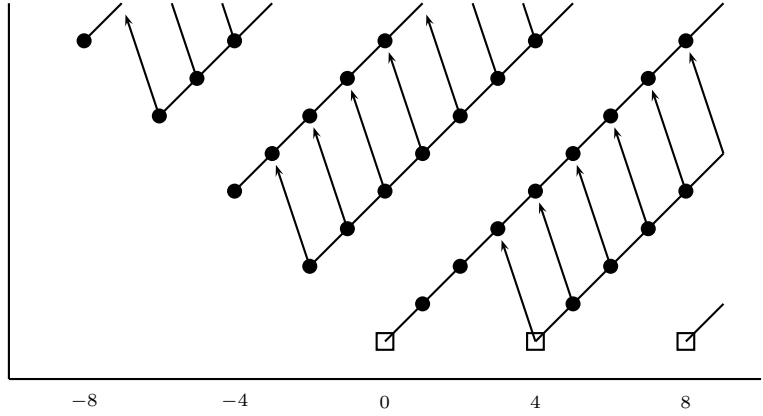
\begin{figure}[htbp!]
\psset{unit=0.5cm}
\begin{pspicture}(-10,-2)(10,10)
\scriptsize

\psline(-10,-1)(10,-1)
\psline(-10,-1)(-10,9)

\rput(0,-1.6){$0$}
\rput(-4,-1.6){$-4$}
\rput(-8,-1.6){$-8$}
\rput(4,-1.6){$4$}
\rput(8,-1.6){$8$}

\psframe(-0.25,-0.25)(0.25,0.25)
\psline(0,0)(1,1)
\pscircle*(1,1){0.2}
\psline(1,1)(2,2)
\pscircle*(2,2){0.2}
\psline(2,2)(3,3)
\pscircle*(3,3){0.2}
\psline(3,3)(4,4)
\pscircle*(4,4){0.2}
\psline(4,4)(5,5)
\pscircle*(5,5){0.2}
\psline(5,5)(6,6)
\pscircle*(6,6){0.2}
\psline(6,6)(7,7)
\pscircle*(7,7){0.2}
\psline(7,7)(8,8)
\pscircle*(8,8){0.2}
\psline(8,8)(9,9)


\pscircle*(-2,2){0.2}
\psline(-2,2)(-1,3)
\pscircle*(-1,3){0.2}
\psline(-1,3)(0,4)
\pscircle*(0,4){0.2}
\psline(0,4)(1,5)
\pscircle*(1,5){0.2}
\psline(1,5)(2,6)
\pscircle*(2,6){0.2}
\psline(2,6)(3,7)
\pscircle*(3,7){0.2}
\psline(3,7)(4,8)
\pscircle*(4,8){0.2}
\psline(4,8)(5,9)


\pscircle*(-4,4){0.2}
\psline(-4,4)(-3,5)
\pscircle*(-3,5){0.2}
\psline(-3,5)(-2,6)
\pscircle*(-2,6){0.2}
\psline(-2,6)(-1,7)
\pscircle*(-1,7){0.2}
\psline(-1,7)(0,8)
\pscircle*(0,8){0.2}
\psline(0,8)(1,9)


\pscircle*(-6,6){0.2}
\psline(-6,6)(-5,7)
\pscircle*(-5,7){0.2}
\psline(-5,7)(-4,8)
\pscircle*(-4,8){0.2}
\psline(-4,8)(-3,9)


\pscircle*(-8,8){0.2}
\psline(-8,8)(-7,9)


\psframe(3.75,-0.25)(4.25,0.25)
\psline(4,0)(5,1)
\pscircle*(5,1){0.2}
\psline(5,1)(6,2)
\pscircle*(6,2){0.2}
\psline(6,2)(7,3)
\pscircle*(7,3){0.2}
\psline(7,3)(8,4)
\pscircle*(8,4){0.2}
\psline(8,4)(9,5)


\psframe(7.75,-0.25)(8.25,0.25)
\psline(8,0)(9,1)

\psline{->}(4,0)(3.1,2.7)
\psline{->}(5,1)(4.1,3.7)
\psline{->}(6,2)(5.1,4.7)
\psline{->}(7,3)(6.1,5.7)
\psline{->}(8,4)(7.1,6.7)
\psline{->}(9,5)(8.1,7.7)

\psline{->}(-2,2)(-2.9,4.7)
\psline{->}(-1,3)(-1.9,5.7)
\psline{->}(0,4)(-0.9,6.7)
\psline{->}(1,5)(0.1,7.7)
\psline{->}(2,6)(1.1,8.7)
\psline(3,7)(2.33,9)
\psline(4,8)(3.67,9)

\psline{->}(-6,6)(-6.9,8.7)
\psline(-5,7)(-5.67,9)
\psline(-4,8)(-4.33,9)

\end{pspicture}

\caption{The $E_{2}$-page of the homotopy fixed points spectral sequence
 for $kgl^{h\Z/2}$}
\label{fig:E^{2}kgl}
\end{figure}
\end{center}

As in Lemma \ref{lem:d3}, $d_3(c) = \tau h_1^3$.  From this,
it follows that
$E_4$ is equal to
\[
\frac{\Z_{2}[\tau, h_{1}, a, b, x]}
{2h_{1}, \tau h_{1}^{3}, \tau h_{1}x, h_{1}a, ax, a^{2}=4b, bx=h_{1}^{4}},
\] 
where the degree of $\tau$ is $(0,0,-1)$, 
the degree of $h_1$ is $(1,1,1)$,
the degree of $x$ is $(-4,4,0)$,
the degree of $a$ is $(4,0,2)$, and
the degree of $b$ is $(8,0,4)$.  
Here $a$ corresponds to $2c$,
$b$ corresponds to $c^{2}$,
and $x$ corresponds to $z^2$.

For dimension reasons, there are no higher differentials,
and $E_{\infty}$ is equal to  $E_{4}$.  Also for dimension reasons,
there are no extensions to resolve in passing from $E_\infty$
to $\pi_{*,*} (kgl^{h\Z/2})$.
A pictorial representation of $E_\infty$ is shown in 
Figure \ref{fig:E^{4}kgl}.
\end{proof}

\begin{center}
\begin{figure}[htbp!]
\psset{unit=0.5cm}
\begin{pspicture}(-10,-2)(10,10)
\scriptsize

\psline(-10,-1)(10,-1)
\psline(-10,-1)(-10,9)

\rput(0,-1.6){$0$}
\rput(-4,-1.6){$-4$}
\rput(-8,-1.6){$-8$}
\rput(4,-1.6){$4$}
\rput(8,-1.6){$8$}

\psframe(-0.25,-0.25)(0.25,0.25)
\psline(0,0)(1,1)
\pscircle*(1,1){0.2}
\psline(1,1)(2,2)
\pscircle*(2,2){0.2}
\psline(2,2)(3,3)
\pscircle(3,3){0.2}
\psline(3,3)(4,4)
\pscircle(4,4){0.2}
\psline(4,4)(5,5)
\pscircle(5,5){0.2}
\psline(5,5)(6,6)
\pscircle(6,6){0.2}
\psline(6,6)(7,7)
\pscircle(7,7){0.2}
\psline(7,7)(8,8)
\pscircle(8,8){0.2}
\psline(8,8)(9,9)


\pscircle*(-4,4){0.2}
\psline(-4,4)(-3,5)
\pscircle(-3,5){0.2}
\psline(-3,5)(-2,6)
\pscircle(-2,6){0.2}
\psline(-2,6)(-1,7)
\pscircle(-1,7){0.2}
\psline(-1,7)(0,8)
\pscircle(0,8){0.2}
\psline(0,8)(1,9)


\pscircle*(-8,8){0.2}
\psline(-8,8)(-7,9)


\psframe(3.75,-0.25)(4.25,0.25)

\psframe(7.75,-0.25)(8.25,0.25)
\psline(8,0)(9,1)


\end{pspicture}

\caption{The $E_{\infty}$-page of the homotopy fixed points spectral
sequence for $kgl^{h\Z/2}$}
\label{fig:E^{4}kgl}
\end{figure}
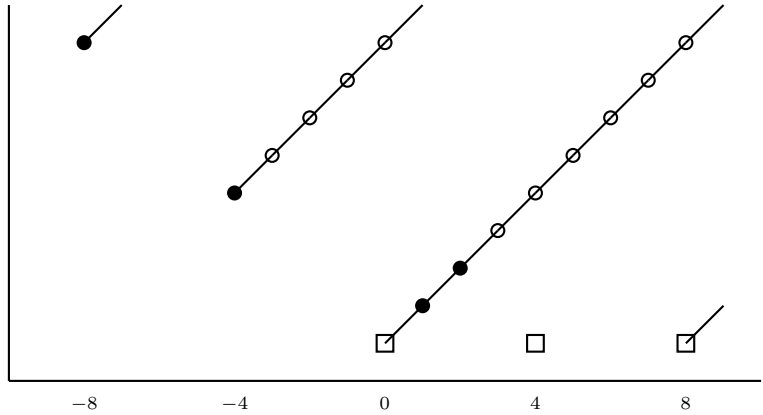
\end{center}

\begin{remark}
We draw the reader's attention to the relation $bx = \tau h_1^4$.
This is a curious difference to the classical case, where the
element $x$ of $\pi_{-2} ku^{h\Z/2}$ annihilates all of the other
generators of $\pi_* ku^{h\Z/2}$.
\end{remark}

\section{The cohomology of $kgl$ and $ko$}
\label{sctn:cohlgy}

In this section, we compute the motivic $\F_2$-cohomology
of $kgl$ and $ko$.

\subsection{The cohomology of $H\Z_2$}
\label{subsctn:HZ-cohlgy}

\begin{definition}
Let $H\Z_2$ be the cofiber of the map $\beta: \Sigma^{2,1} kgl \map kgl$.
\end{definition}

\begin{remark}
As its name suggests, $H\Z_2$ represents motivic cohomology 
with $\Z_2$-coefficients.  We will not use this fact.  Everything
that we need to know about $H\Z_2$ comes from 
its definition as a cofiber.
\end{remark}

\begin{lemma}
The cofiber of the map $2: H\Z_2 \map H\Z_2$ is $H\F_2$.
\end{lemma}

\begin{proof}
As $\Z_2[\tau]$-modules, the homotopy groups of the cofiber are isomorphic
to the homotopy groups of $H\F_2$.  By Proposition \ref{prop:EM-unique},
the cofiber is weakly equivalent to $H\F_2$.
\end{proof}

We will write $p: H\Z_2 \map H\F_2$ for the map from $H\Z_2$ to the cofiber
of $2$.  Also, we will write $\delta: H\F_2 \map \Sigma H\Z_2$ for
the boundary map of the cofiber sequence of $2$.  We will make use of 
the cofiber sequence
\[
\xymatrix{
H\Z_2 \ar[r]^2 & H\Z_2 \ar[r]^p & H\F_2 \ar[r]^\delta & \Sigma H\Z_2.  }
\]

We shall freely substitute the motivic Steenrod algebra $\SA$ for 
$H^{*,*} H\F_2$,
since they are equal by definition.

\begin{lemma}
\label{lem:HZ-delta-p}
The composition $\delta^* p^*: \SA \map \Sigma^{-1} \SA$
is equal to right multiplication by $\Sq{1}$.
\end{lemma}

\begin{proof}
Since $\delta^* p^*$ is an $\SA$-module map, it suffices to 
compute $\delta^* p^* (1)$.

Consider the diagram
\[
\xymatrix{
S^{0,0} \ar[r]^{p'}\ar[d] & M2 \ar[d]\ar[r]^{\delta'} & S^{1,0} \ar[d] \\
H\Z_2 \ar[r]_p & H\F_2 \ar[r]_\delta & \Sigma H\Z_2,  }
\]
where both rows are cofiber sequences induced by 
the map $2$.  Here $M2$ is the motivic mod 2 Moore spectrum.
Recall that $H^{*,*} M2$ is free over $\M_2$ on two generators
$x$ and $y$ of degrees $(0,0)$ and $(1,0)$, and $\Sq{1} x = y$ because
$\Sq{1}$ is the same as the Bockstein \cite[Section 9]{V2}.  It follows
that the map $\delta'^* p'^*:H^{*,*} M2 \map H^{*,*} \Sigma^{-1} M2$
takes $x$ to $y$.

The elements $1$ and $\Sq{1}$ of $\SA$ map
to $x$ and $y$ respectively in $H^{*,*} M2$.  
A diagram chase now shows that $\delta^* p^*$ takes $1$ to $\Sq{1}$.  
\end{proof}

\begin{theorem}
\label{thm:HZ-cohlgy}
The motivic $\F_2$-cohomology $H^{*,*} H\Z_2$ of $H\Z_2$
is equal to $\SA//\SE(0)$.
\end{theorem}

\begin{proof}
The composition $\delta p$ is null-homotopic, so $p^* \delta^*$ is zero.
This implies that $p^* \delta^* p^*$ is zero.  
By Lemma \ref{lem:HZ-delta-p}, this shows that $p^*$ annihilates
the left ideal generated by $\Sq{1}$.  Hence $p^*$ extends 
to a map
$\overline{p}^*: \SA//\SE(0) \map H^{*,*} H\Z_2$.
This gives us a commutative diagram
\[
\xymatrix{
0 \ar[r] & H^{*,*} \Sigma H\Z_2 \ar[d]_{\Sigma\overline{p}^*}\ar[r]^{\delta^*} &
    H^{*,*} H\F_2 \ar[d]_{\cong}\ar[r]^{p^*} & 
    H^{*,*} H\Z_2 \ar[d]_{\overline{p}^*}\ar[r] & 0 \\
0 \ar[r] & \Sigma \SA//\SE(0) \ar[r] & \SA \ar[r] & \SA// \SE(0) \ar[r] & 0.  }
\]
The top row is exact because 
the map $H^{*,*} H\Z_2 \map H^{*,*}H\Z_2$ induced by $2$ is zero.
The bottom row is exact because of Lemma \ref{lem:subalg-exact}.

The right square commutes by definition of $\overline{p}^*$, and
the left square commutes by Lemma \ref{lem:HZ-delta-p}.
It follows that $\overline{p}^*$ is an isomorphism.
\end{proof}

\subsection{The cohomology of $kgl$}
\label{subsctn:kgl-cohlgy}

Having computed the cohomology of $H\Z_2$, we will now exploit the
cofiber sequence
\[
\xymatrix{
\Sigma^{2,1} kgl \ar[r]^\beta & kgl \ar[r]^p & H\Z_2 \ar[r]^\delta &
\Sigma^{3,1} kgl }
\]
in order to compute the cohomology of $kgl$.  Beware that 
$p$ and $\delta$ are different than (but analogous to) the 
maps of the same name in Section \ref{subsctn:HZ-cohlgy}.

Because of Theorem \ref{thm:HZ-cohlgy}, we shall freely substitute
$\SA//\SE(0)$ for $H^{*,*} H\Z_2$.

\begin{lemma}
\label{lem:kgl-delta-p}
The composition $\delta^* p^*: \SA//\SE(0) \map \Sigma^{-1} \SA//\SE(0)$
is equal to right multiplication by $\SQ{1}$.
\end{lemma}

\begin{proof}
Since $\delta^* p^*$ is an $\SA$-module map,
it suffices to show that $\delta^* p^* (1) = \SQ{1}$.

Consider the classical case.  We have a cofiber sequence
\[
\xymatrix{
\Sigma^2 ku \ar[r]^\beta & ku \ar[r]^{p'} & H\Z \ar[r]^{\delta'} & \Sigma^3 ku. }
\]
It follows from the arguments of \cite[p.\ 366]{Adams} that
$\delta'^* p'^*(1) = Q_1$.
Since our motivic computation must be
compatible with the classical computation under topological realization,
it follows that $\delta^* p^*$ takes
$1$ to $\SQ{1}$.  
\end{proof}

\begin{theorem}
\label{thm:kgl-cohlgy}
The motivic $\F_2$-cohomology $H^{*,*} kgl$ of $kgl$
is equal to $\SA//\SE(1)$.
\end{theorem}

\begin{proof}
The composition $\delta p$ is null-homotopic, so $p^* \delta^*$ is zero.
This implies that $p^* \delta^* p^*$ is zero.  
By Lemma \ref{lem:kgl-delta-p}, this shows that $p^*$ annihilates
the left ideal generated by $\SQ{1}$.  Hence $p^*$ extends 
to a map
$\overline{p}^*: \SA//\SE(1) \map H^{*,*} kgl$.
This gives us a commutative diagram
\[
\xymatrix{
H^{*,*} \Sigma kgl \ar[r]^{\Sigma\beta^*} & 
    H^{*,*} \Sigma^{3,1} kgl\ar[d]_{\Sigma^{3,1}\overline{p}^*}\ar[r]^{\delta^*}&
    H^{*,*} H\Z_2 \ar[d]_{\cong}\ar[r]^{p^*} & 
    H^{*,*} kgl \ar[d]_{\overline{p}^*} \ar[r]^{\beta^*} & 
    H^{*,*} \Sigma^{2,1} kgl \\
0 \ar[r] & \Sigma^{3,1} \SA//\SE(1) \ar[r] & \SA//\SE(0) \ar[r] & 
  \SA// \SE(1) \ar[r] & 0.  }
\]
The bottom row is exact because of Lemma \ref{lem:subalg-exact}.
The right square commutes by definition of $\overline{p}^*$, and
the left square commutes by Lemma \ref{lem:kgl-delta-p}.

We already know that $\overline{p}^*$ is surjective because the
right square commutes.  We shall prove by induction that
$\overline{p}^*$ is an isomorphism.
The base case occurs in bidegrees $(a,b)$ with $a<0$.
If $a<0$, then 
$H^{a,b} kgl$ is zero by Lemma \ref{lem:pos-neg}, and 
$\SA//\SE(1)$ is also zero by definition.

Now suppose that $\overline{p}^*$ is an isomorphism in bidegrees
$(a,b)$ for $a < n$.  Then
$\Sigma^{3,1} \overline{p}^*$ is an isomorphism in bidegrees
$(a,b)$ for $a< n+3$.  
This implies that $\delta^*$ is injective in bidegrees $(a,b)$
for $a<n+3$, so $\Sigma \beta^*$ is zero in the same bidegrees.
In other words, $\beta^*$ is zero in bidegrees $(a,b)$ with $a<n+4$.
We have now shown that the top sequence in the above diagram
splits in bidegrees $(a,b)$ with $a<n+3$.  It follows
that $\overline{p}^*$ is an isomorphism in bidegrees $(a,b)$ with
$a < n+3$.  This completes the induction step.
\end{proof}

\subsection{The cohomology of $ko$}
\label{subsctn:ko-cohlgy}

In order to compute the cohomology of $ko$, we will use an argument that
is very similar to the argument of Section \ref{subsctn:kgl-cohlgy}.

Recall the motivic Hopf map $\eta$ in $\pi_{1,0}$.  The shortest
description of $\eta$ is the projection $\AAA^2-0 \map \p^1$,
because $\AAA^2-0$ is a model for $S^{3,2}$ and $\p^1$ is a model 
for $S^{2,1}$.

\begin{lemma}
\label{lem:cofiber-eta}
The cofiber of the map $\Sigma^{1,1} ko \map ko$ induced by $\eta$
is weakly equivalent to $kgl$.
\end{lemma}

\begin{proof}
Let $C$ be the cofiber of $\Sigma^{1,1} ko \map ko$.  
There is a map $f:ko \map kgl$ because of formal properties of
homotopy fixed points spectra.  Consider the diagram
\[
\xymatrix{
ko \Smash S^{1,1} \ar[dr]\ar[r]^{\id \Smash \eta} 
  & ko \Smash ko \ar[r]^{f\Smash f}\ar[d] & kgl \Smash kgl \ar[d] \\
& ko \ar[r]_f & kgl,
}
\]
where the vertical maps are the multiplication maps of the motivic
ring spectra $ko$ and $kgl$.  
The composition across the top is
the smash product of $f$ with a map $S^{1,1} \map kgl$,
which must be zero since $\pi_{1,1} kgl$ is zero.
This shows that the composition
\[
\Sigma^{1,1} ko \map ko \map kgl
\]
is zero, so $f$ extends to a map $C \map kgl$.

Having constructed the map $C \map kgl$, it remains to show that
it is an isomorphism on homotopy groups.  Then we can finish
with Proposition \ref{prop:pi-cellular}.

Use the computation of $\pi_{*,*} ko$ and the long exact sequence
in homotopy groups to compute $\pi_{*,*} C$.  Analysis of the long
exact sequence leaves one ambiguity, which can be resolved by
showing that the element $\beta$ of $\pi_{2,1} kgl$
is contained in the Toda bracket $\langle i, \eta, 2\rangle$,
where $i: S^{0,0} \map kgl$ is the unit map.  This computation can be
made by comparing to the classical computation via topological realization.
\end{proof}

\begin{remark}
The same argument shows that $KGL$ is the cofiber of
$\eta: KO \map KO$.
\end{remark}

Having computed the cohomology of $kgl$, we will now exploit the
cofiber sequence
\[
\xymatrix{
\Sigma^{1,1} ko \ar[r]^\eta & ko \ar[r]^p & kgl \ar[r]^\delta &
\Sigma^{2,1} ko }
\]
in order to compute the cohomology of $ko$.  Beware that 
$p$ and $\delta$ are different than (but analogous to) the 
maps of the same name in Sections \ref{subsctn:HZ-cohlgy}
and \ref{subsctn:kgl-cohlgy}.

Because of Theorem \ref{thm:kgl-cohlgy}, we shall freely substitute
$\SA//\SE(1)$ for $H^{*,*} kgl$.

\begin{lemma}
\label{lem:ko-delta-p}
The composition $\delta^* p^*: \SA//\SE(1) \map \Sigma^{-1} \SA//\SE(1)$
is equal to right multiplication by $\Sq{2}$.
\end{lemma}

\begin{proof}
Since $\delta^* p^*$ is an $\SA$-module map,
it suffices to show that $\delta^* p^* (1) = \Sq{2}$.

Consider the diagram
\[
\xymatrix{
S^{0,0} \ar[r]^{p'}\ar[d] & C\eta \ar[d]\ar[r]^{\delta'} & S^{2,1} \ar[d] \\
ko \ar[r]_p & kgl \ar[r]_\delta & \Sigma^{2,1} ko,  }
\]
where both rows are cofiber sequences induced by 
the map $\eta$.  Here $C\eta$ is the suspension spectrum of
$\p^2$.
Recall that $H^{*,*} C\eta$ is free over $\M_2$ on two generators
$x$ and $y$ of degrees $(0,0)$ and $(2,1)$, and $\Sq{2} x = y$ because
of the cup product structure on $H^{*,*} \p^2$.
It follows
that the map $\delta'^* p'^*:H^{*,*} C\eta \map H^{*,*} C\eta$
takes $x$ to $y$.

The elements $1$ and $\Sq{2}$ of $\SA$ map
to $x$ and $y$ respectively in $H^{*,*} C\eta$.  
A diagram chase now shows that $\delta^* p^*$ takes $1$ to $\Sq{2}$.  
\end{proof}

\begin{theorem}
\label{thm:ko-cohlgy}
The motivic $\F_2$-cohomology $H^{*,*} ko$ of $ko$
is equal to $\SA//\SA(1)$.
\end{theorem}

\begin{proof}
The proof is essentially identical to the proof
of Theorem \ref{thm:kgl-cohlgy}.  In the diagram,
\[
\xymatrix{
H^{*,*} \Sigma ko \ar[r]^{\Sigma\eta^*} & 
    H^{*,*} \Sigma^{2,1} ko\ar[d]_{\Sigma^{2,1}\overline{p}^*}\ar[r]^{\delta^*}&
    H^{*,*} kgl \ar[d]_{\cong}\ar[r]^{p^*} & 
    H^{*,*} ko \ar[d]_{\overline{p}^*} \ar[r]^{\eta^*} & 
    H^{*,*} \Sigma^{1,1} ko \\
0 \ar[r] & \Sigma^{2,1} \SA//\SA(1) \ar[r] & \SA//\SE(1) \ar[r] & 
  \SA// \SA(1) \ar[r] & 0  }
\]
one can prove by induction that $\overline{p}^*$ is an isomorphism.
\end{proof}

\section{Computations of $ko$-homology}
\label{sctn:Ext-compute}

The first goal of this section is to establish an Adams spectral
sequence for computing $ko$-homology.  We identify the $E_2$-term
of this spectral sequence in terms of $\Ext$ groups over $\SA(1)$.
The remainder of the section is dedicated to computing
these $\Ext$ groups over $\SA(1)$ for various $\SA(1)$-modules
of interest, i.e., $E_2$-terms of the spectral sequence.
In all of the cases that we study below,
all differentials are trivial for simple algebraic reasons.
The interested reader can reconstruct $ko$-homology groups from
our computations.

\begin{lemma}
\label{lem:ko-smash-DQ}
For $0 \leq m \leq n \leq \infty$, there is an isomorphism
\[
H^{*,*} \left( ko \Smash \frac{DQ_n}{DQ_m} \right)
\cong
H^{*,*} ko \otimes_{\M_2} H^{*,*} \left( \frac{DQ_n}{DQ_m} \right).
\]
\end{lemma}

\begin{proof}
By a K\"unneth theorem for motivic cohomology \cite[Theorem 8.6]{DDDI2},
the lemma holds for $n < \infty$ because $DQ_n / DQ_m$ is a finite complex.
Note that the higher $\Tor$ terms
\[
\Tor^{\M_2} \left( H^{*,*} ko, H^{*,*} \left( \frac{DQ_n}{DQ_m} \right) \right)
\]
vanish because $H^{*,*} ko$ is free over $\M_2$ by Theorem \ref{thm:ko-cohlgy}.

It remains to consider the case $n = \infty$.  Recall that
$DQ_\infty / DQ_m$ is equal to $\colim_k DQ_k / DQ_m$ and that
$ko \Smash \frac{DQ_\infty}{DQ_m}$ is equal to
$\colim_k ko \Smash \frac{DQ_k}{DQ_m}$.
Observe that both
$\lim_k^1 H^{*,*} \left( \frac{DQ_k}{DQ_m} \right)$ and
$\lim_k^1 H^{*,*} \left( ko \Smash \frac{DQ_k}{DQ_m} \right)$
vanish.  This follows from the fact that for fixed $p$ and $q$, the groups 
$H^{p,q} \left( \frac{DQ_k}{DQ_m} \right)$ and
$H^{p,q} \left( ko \Smash \frac{DQ_k}{DQ_m} \right)$ 
do not depend on $k$, as long as $k$ is sufficiently large.

Using $\lim^1$ short exact sequences and the first paragraph, we have the chain
\begin{eqnarray*}
H^{*,*} \left( ko \Smash \frac{DQ_\infty}{DQ_m} \right) & \cong &
\lim_k H^{*,*} \left( ko \Smash \frac{DQ_k}{DQ_m} \right) \\
& \cong & 
\lim_k H^{*,*} ko \otimes_{\M_2} H^{*,*} \left( \frac{DQ_k}{DQ_m} \right) \\
& \cong & 
H^{*,*} ko \otimes_{\M_2} \lim_k H^{*,*} \left( \frac{DQ_k}{DQ_m} \right) \\
& \cong & 
H^{*,*} ko \otimes_{\M_2} H^{*,*} \left( \frac{DQ_\infty}{DQ_m} \right)
\end{eqnarray*}
of isomorphisms.
\end{proof}

\begin{remark}
It is possible to prove that
$H^{*,*} (ko \Smash X)$ is isomorphic to
$H^{*,*} ko \otimes_{\M_2} H^{*,*} X$ for a larger class of $X$
than in Lemma \ref{lem:ko-smash-DQ}.  We have avoided this generality
for sake of simplicity.
\end{remark}

\begin{theorem}
\label{thm:ASS}
Let $0 \leq m \leq n \leq \infty$.
There is a spectral sequence
\[
\Ext_{\SA(1)}\left( H^{*,*} \left( \frac{DQ_n}{DQ_m} \right), \M_2 \right) 
  \Rightarrow ko_{*,*} X.
\]
\end{theorem}

\begin{proof}
By \cite[Lemma 7.13]{DI3}, we have an Adams spectral sequence
\[
\Ext_{\SA} \left( H^{*,*} \left( ko \Smash \frac{DQ_n}{DQ_m} \right), 
  H^{*,*} S^{0,0} \right) \Rightarrow ko_{*,*} X.
\]
By Lemma \ref{lem:ko-smash-DQ} and Theorem \ref{thm:ko-cohlgy},
$H^{*,*} \left( ko \Smash \frac{DQ_n}{DQ_m} \right)$ is isomorphic to
$\SA//\SA(1) \otimes_{\M_2} H^{*,*} \left( \frac{DQ_n}{DQ_m} \right)$.
A standard change of rings finishes the proof.
\end{proof}

\begin{remark}
Similarly, because of Theorems \ref{thm:HZ-cohlgy} and \ref{thm:kgl-cohlgy},
there are Adams spectral sequences
\[
\Ext_{\SE(0)}(H^{*,*} X, \M_2) \Rightarrow (H\Z_2)_{*,*} X
\]
and
\[
\Ext_{\SE(1)}(H^{*,*} X, \M_2) \Rightarrow kgl_{*,*} X.
\]
\end{remark}

The rest of this section is dedicated to computing
$\Ext$ groups over $\SA(1)$ for various $\SA(1)$-modules
of interest. 

\begin{definition}
\label{def:AdamsGrading} 
For $\SA(1)$-modules $M$ and $N$, an element in $\Ext_{\SA(1)}^{n}(M,N)$ of internal bidegree $(a,b)$ has Adams tridegree $(a-n,n,b)$ and Adams bidegree $(a-n,n)$.
\end{definition}

\begin{theorem} 
\label{theo:Ext(M2,M2)ring} 
The ring $\Ext^{*}_{\SA(1)}(\M_2,\M_2)$ is the $\M_2$-algebra 
given by the following generators and relations:
\begin{center}
\begin{tabular}{|l|l|}
\hline
generator & Adams tridegree \\
\hline
$h_0$ & $(0,1,0)$ \\
$h_1$ & $(1,1,1)$ \\
$\alpha$ & $(4,3,2)$ \\
$\beta$ & $(8,4,4)$ \\
\hline
\end{tabular}
\hspace{5em}
\begin{tabular}{|l|}
\hline
relations \\
\hline
$h_0 h_1 = 0$ \\
$\tau h_1^3 = 0$ \\
$h_1 \alpha = 0$ \\
$\alpha^2 = h_0^2 \beta$ \\
\hline
\end{tabular}
\end{center}
\end{theorem}

\begin{proof}
The shortest proof is to use the motivic May spectral sequence, as in
\cite{DI3} and \cite{DI-A(2)}.

Alternatively, one can construct an explicit $\SA(1)$-resolution of $\M_2$
to compute the additive structure of the $\Ext$ groups, and then use
explicit cocycles to find Yoneda products.
\end{proof}

Figure \ref{fig:AdamsM2} is a pictorial representation of 
the previous theorem.  

\begin{remark}
Recall the calculation of $\pi_{*,*}(ko)$ from Theorem \ref{thm:main2}.
The calculation in Theorem \ref{theo:Ext(M2,M2)ring} is 
the associated graded of $\pi_{*,*} (ko)$, filtered by powers of $2$.
\end{remark}

In order to simplify notation, we write $\E$ for
$\Ext_{\SA(1)}(\M_2,\M_2)$.
We will describe $\Ext$ groups of various $\SA(1)$-modules
as modules over the ring
$\Ext_{\SA(1)}(\M_2,\M_2)$.

From now on, we will use \textit{reduced} motivic cohomology 
with coefficients in $\F_{2}$.  The point is that unreduced
cohomology splits as an $\SA(1)$-module into reduced cohomology
plus a copy of $\M_2$.  From the perspective of $\SA(1)$-module
theory, the extra copy of $\M_2$ clutters the calculations needlessly.

\begin{definition} 
Let $R$ be the $\SA(1)$-module generated by $x_{i}$ for $i\geq 0$,
where the degree of $x_i$ is $(4i-1,2i)$,
subject to the relations $\Sq{2,1,2} x_i = \Sq{1} x_{i+1}.$ 
\end{definition}
 
Figure \ref{fig:R} gives a pictorial representation of $R$, with
notation as in the previous figures.

\begin{figure}[htbp!]
\begin{center}
\psset{unit=0.5cm}
\begin{pspicture}(-2,-1)(2,7)

\pscircle(0,-1){0.3}
\pscircle(0,1){0.3}
\pscircle(0,2){0.3}
\pscircle(0,3){0.3}
\pscircle(0,4){0.3}
\pscircle(0,5){0.3}
\pscircle(0,6){0.3}

\psline(0,1)(0,2)
\psline(0,3)(0,4)
\psline(0,5)(0,6)

\psbezier(0,-1)(0.7,-0.3)(0.7,0.3)(0,1)
\psbezier(0,2)(-0.7,2.7)(-0.7,3.3)(0,4)
\psbezier(0,3)(0.7,3.7)(0.7,4.3)(0,5)

\rput(-2,-1){$x_0$}
\rput(-2,3){$x_1$}

\rput(2,-1){$(-1,0)$}
\rput(2,1){$(1,1)$}
\rput(2,2){$(2,1)$}
\rput(2,3){$(3,2)$}
\rput(2,4){$(4,2)$}
\rput(2,5){$(5,3)$}
\rput(2,6){$(6,3)$}

\rput(0,7){$\vdots$}

\end{pspicture}
\end{center}
\caption{The $\SA(1)$-module $R$}\label{fig:R}
\end{figure}
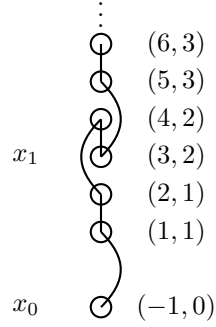

\begin{theorem}
As an $\E$-module,
$\Ext_{\SA(1)}^{*}(R,\M_2)$ is given by the following
generators and relations:
\begin{center}
\begin{tabular}{|l|l|}
\hline
generator & Adams tridegree \\
\hline
$r_{4i-1}$ for $i \geq 0$ & $(4i-1,0,2i)$ \\
\hline
\end{tabular}
\hspace{5em}
\begin{tabular}{|l|}
\hline
relations \\
\hline
$h_1 r_{4i-1} = 0$ for $i \geq 0$ \\
$\alpha r_{4i-1} = h_0^3 r_{4i+3}$ for $i \geq 0$ \\
$\beta r_{4i-1} = h_0^4 r_{4i+7}$ for $i \geq 0$ \\
\hline
\end{tabular}
\end{center}
\end{theorem}

\begin{proof}
It is straightforward to write down a free $\SA(1)$-resolution of $R$.
The $\E$-module structure comes from explicit computations with
cocycles in low dimensions.
\end{proof}

Figure \ref{fig:Ext(R,M2)} is a pictorial representation
of $\Ext_{\SA(1)}^{*}(R,\M_2)$.

\begin{theorem}
\label{theo:DQinf} 
As an $\E$-module, 
$\Ext_{\SA(1)}^{*}(\widetilde H^{*,*}(DQ_{\infty}),\M_2)$ is given
by the following generators and relations:
\begin{center}
\begin{tabular}{|l|l|}
\hline
generator & Adams tridegree \\
\hline
$k$ & $(1,0,1)$ \\
$r_{4i-1}$ for $i \geq 1$ & $(4i-1,0,2i)$ \\
\hline
\end{tabular}
\hspace{5em}
\begin{tabular}{|l|}
\hline
relations \\
\hline
$h_0 k = 0$ \\
$\alpha k = 0$ \\
$h_0^2 r_3 = \tau h_1^2 k$ \\
$h_0^4 r_7 = 0$ \\
$\alpha r_{4i-1} = h_0^3 r_{4i+3}$ for $i \geq 1$ \\
$\beta r_{4i-1} = h_0^4 r_{4i+7}$ for $i \geq 1$ \\
\hline
\end{tabular}
\end{center}
\end{theorem}

\begin{proof}
There is a short exact sequence
\[
\widetilde H^{*,*}(DQ_{\infty})\rightarrow R \map
\Sigma^{-1,0} \M_2,
\]
which yields a long exact sequence in $\Ext$ groups.
The maps of the long exact sequence are completely determined
by $\E$-linearity and direct calculation in homological degree $0$.
The long exact sequence 
tells us most of what we want to know, but it 
leaves one ambiguity in the $\E$-module structure of 
$\Ext_{\SA(1)}^{*}(\widetilde H^{*,*}(DQ_{\infty}),\M_2)$.
Namely, it is not immediately clear whether
$\tau h_1^2 k$ equals zero or $h_0^2 r_1$.
This ambiguity 
can be resolved by explicit calculations with cocycles of homological
degree $2$.
\end{proof}

Figure \ref{fig:ExtDQinf} is a pictorial representation of the
above result.

\begin{theorem}
\label{theo:n=0}
Let $n$ be a positive integer that is congruent to $0$ modulo $4$.
As an $\E$-module,
$\Ext_{\SA(1)}^{*}(\widetilde H^{*,*}(DQ_{n}),\M_2)$ 
is given by the following generators and relations:
\begin{center}
\begin{tabular}{|l|l|}
\hline
generator & Adams tridegree \\
\hline
$k$ & $(1,0,1)$ \\
$r_{4i-1}$ for $1 \leq i \leq \frac{n}{4}$ & $(4i-1,0,2i)$ \\
\hline
\end{tabular}
\hspace{1em}
\begin{tabular}{|l|}
\hline
relations \\
\hline
$h_0 k = 0$ \\
$\alpha k = 0$ \\
$h_0^2 r_3 = \tau h_1^2 k$ \\
$h_0^4 r_7 = 0$ \\
$h_1 r_{4i-1} = 0$ if $1 \leq i < \frac{n}{4}$ \\
$\alpha r_{4i-1} = h_0^3 r_{4i+3}$ if $1 \leq i < \frac{n}{4}$ \\
$\beta r_{4i-1} = h_0^4 r_{4i+7}$ if $1 \leq i < \frac{n}{4} - 1$ \\
$\beta r_{n-5} = h_0 \alpha r_{n-1}$ \\
\hline
\end{tabular}
\end{center}
\end{theorem}

\begin{proof}
Consider the short exact sequence
\[
\widetilde H^{*,*}\left(\frac{DQ_{\infty}}{DQ_{n}}\right) \map
\widetilde H^{*,*}(DQ_{\infty}) \map
\widetilde H^{*,*}(DQ_{n}).
\]
Since $n$ is congruent to $0$ modulo $4$, the first term is isomorphic to 
$\Sigma^{n,\frac{n}{2}}\widetilde H^{*,*}(DQ_{\infty})$.  
Therefore, as an $\E$-module, 
$\Ext_{\SA(1)}^{*}(\widetilde H^{*,*} (\frac{DQ_{\infty}}{DQ_{n}}),\M_2)$
is equal to a shifted copy of
$\Ext_{\SA(1)}^{*}(\widetilde H^{*,*}(DQ_{\infty}),\M_2)$,
which we computed in Theorem \ref{theo:DQinf}.

The maps in the associated long exact of $\Ext$ groups are
entirely determined by $\E$-linearity and explicit computation
in homological degree $0$.
The long exact sequence tells us most of what we need.
The only ambiguity is whether
$h_1 r_{n-1}$ is zero or not zero.
Computations with explicit cocycles in homological degree 1 show
that it is not zero.
\end{proof}

Figure \ref{fig:DQn=0} is a pictorial representation of the above result.

\begin{theorem}
\label{theo:DQ2}
As an $\E$-module, 
$\Ext_{\SA(1)}^{*}(\widetilde H^{*,*}(DQ_{2}),\M_2)$ 
is given by the following generators and relations:
\begin{center}
\begin{tabular}{|l|l|}
\hline
generator & Adams tridegree \\
\hline
$x$ & $(1,0,1)$ \\
$y$ & $(3,1,2)$ \\
\hline
\end{tabular}
\hspace{5em}
\begin{tabular}{|l|}
\hline
relations \\
\hline
$h_0 x = 0$ \\
$h_0 y = \tau h_1^2 x$ \\
$\alpha x = \tau h_1^2 y$ \\
$\alpha y = 0$ \\
\hline
\end{tabular}
\end{center}
\end{theorem}

\begin{proof}
The short exact sequence
\[
\Sigma^{2,1} \M_{2} \map 
\widetilde H^{*,*}(DQ_{2}) \map 
\Sigma^{1,1} \M_{2}
\]
yields a long exact sequence in $\Ext$ groups.
The maps in the long exact sequence are determined by $\E$-linearity,
together with explicit computations in homological degrees $0$ and $1$.
The long exact sequence tells us most of what we want.
The only ambiguity is whether $\tau h_1^2 y$ is zero or non-zero.
Explicit computations with cocycles in low dimensions
shows that it is non-zero.
\end{proof}

Figure \ref{fig:DQ2} is a pictorial representation of the above result.

\begin{theorem} 
\label{theo:n=2}
Let $n$ be congruent to $2$ modulo $4$, and let $n > 2$.
As an $\E$-module,
$\Ext_{\SA(1)}^{*}(\widetilde H^{*,*} (DQ_{n}),\M_2)$ 
is given by the following generators and relations:
\begin{center}
\begin{tabular}{|l|l|}
\hline
generator & Adams tridegree \\
\hline
$k$ & $(1,0,1)$ \\
$z$ & $(n+1,1,\frac{n+2}{2})$ \\
$r_{4i-1}$ for $1 \leq i \leq \frac{n-2}{4}$ & $(4i-1,0,2i)$ \\
\hline
\end{tabular}
\hspace{1em}
\begin{tabular}{|l|}
\hline
relations \\
\hline
$h_{0}k=0$ \\
$\alpha k=0$ \\
$h_{0}^{2}r_{3}=\tau h_{1}^{2}k$ \\
$h_{0}^{4}r_{7}=0$ \\
$h_{1}r_{4i-1}=0$ if $1\leq i \leq \frac{n-2}{4}$\\
$\alpha r_{4i-1}=h_{0}^{3}r_{4i+3}$ if $1 \leq i < \frac{n-2}{4}$ \\
$\beta r_{4i-1}=h_{0}^{4}r_{4i+7}$ if $1 \leq i < \frac{n-6}{4}$ \\
$\alpha r_{n-3}=h_{0}^{2}z$ \\
$\beta r_{n-7}=h_{0}^{3}z$ \\
$\beta r_{n-3} = \alpha z$ \\
\hline
\end{tabular}
\end{center}
\end{theorem}

\begin{proof} 
The quotient $\widetilde H^{*,*}(\frac{DQ_{n+2}}{DQ_{n}})$ is isomorphic to 
$\Sigma^{n,\frac{n}{2}} \widetilde H^{*,*}(DQ_{2}).$ 
The short exact sequence 
\[
\Sigma^{n,\frac{n}{2}}\widetilde H^{*,*}(DQ_{2}) \map
\widetilde H^{*,*}(DQ_{n+2}) \map \widetilde H^{*,*}(DQ_{n})
\]
yields a long exact sequence in $\Ext$ groups.  The maps in this sequence
are determined by $\E$-linearity and explicit computations in 
homological degree $0$.

The long exact sequence gives us most of what we want.
The only ambiguity is whether $h_1 z$ is zero or non-zero.
Computations with explicit cocycles in homological dimension $2$
shows that $h_1 z$ is non-zero.
\end{proof}

Figure \ref{fig:DQn=2} is a pictorial representation of the previous
result.

\begin{remark}
When  $n$ is conrguent to $3$ modulo $4$,
the $\SA(1)$-module $\widetilde H^{*,*}(DQ_{n})$ splits as
\[
\widetilde H^{*,*}(DQ_{n-1}) 
\bigoplus \Sigma^{n,\frac{n+1}{2}}\M_{2}.
\]
Therefore this case reduces to the case
when $n$ is congruent to $2$ modulo $4$.
\end{remark}

\begin{theorem} 
\label{theo:n=1} 
Let $n$ be congruent to $1$ modulo $4$, and let $n>1$.
As an $\E$-module,
$\Ext_{\SA(1)}^{*}(\widetilde H^{*,*}(DQ_{n}), \M_2)$ 
is given by the following generators and relations:
\begin{center}
\begin{tabular}{|l|l|}
\hline
generator & Adams tridegree \\
\hline
$k$ & $(1,0,1)$ \\
$r_{4i-1}$ for $1 \leq i \leq \frac{n-1}{4}$ & $(4i-1,0,2i)$ \\
$v$ & $(n,1,\frac{n+1}{2})$ \\
$u$ & $(n+2,2,\frac{n+3}{2})$ \\
$t$ & $(n+4,3,\frac{n+5}{2})$ \\
\hline
\end{tabular}
\hspace{5em}
\begin{tabular}{|l|}
\hline
relations \\
\hline
$h_{0}k=0$ \\
$\alpha k=0$ \\
$h_0^2 r_3 = \tau h_1^2 k$ \\
$h_0^4 r_7 = 0$ \\
$h_{1}r_{4i-1}=0$ if $1 \leq i \leq \frac{n-1}{4}$ \\
$\alpha r_{4i-1}=h_{0}^{3}r_{4i+3}$ if $1 \leq i < \frac{n-1}{4}$ \\
$\beta r_{4i-1}=h_{0}^4 r_{4i+7}$ if $1 \leq i < \frac{n-5}{4}$ \\
$\alpha r_{n-2}=h_{0}u$ \\
$\beta r_{n-6} = h_0^2 u$ \\
$h_{1}v=0$ \\
$h_{1}u=0$ \\
$h_{1}t=0$ \\
$\alpha v=h_{0} t$ \\
$\alpha u=h_{0}\beta r_{n-2}$. \\
$\alpha t=h_{0}\beta v$ \\
\hline
\end{tabular}
\end{center} 

\end{theorem}

\begin{proof}
The short exact sequence
\[
\Sigma^{n+1,\frac{n+1}{2}}\M_2 \map 
\widetilde H^{*,*}(DQ_{n+1}) \map 
\widetilde H^{*,*}(DQ_{n})
\]
yields a long exact sequence in $\Ext$ groups.  The maps in this 
sequence are determined by $\E$-linearity and explicit computations
in homological degrees $0$ and $1$.

The long exact sequence gives us most of what we want to know.
The only ambiguity is whether $h_0 t$ is zero or non-zero.
As usual, this can be resolved by explicit computations with
cocycles in low dimensions.  The result is that $h_0 t$ is non-zero.
\end{proof}

\appendix

\newpage

\section{Adams charts}

This appendix contains Adams charts for many of the computations
made earlier in Section \ref{sctn:Ext-compute}.

We use the following notation in all of the charts.
Solid dots represent copies of $\M_2$, while open circles represent
copies of $\M_2 / \tau$.  

Vertical lines indicate multiplications by $h_0$,
and lines of slope 1 indicate multiplications by $h_1$.
Dotted vertical lines indicate that 
$h_0$ times a generator equals $\tau$ times a generator.
In Figure \ref{fig:DQ2}, the diagonal dotted lines indicate that
$\alpha$ times a generator equals $\tau$ times a generator.

The horizontal and vertical coordinates indicate the Adams bidegree.
The weights of some elements are shown in parentheses.

Figure \ref{fig:DQn=0} is representative of the computation
of $\Ext^*_{\SA(1)}(\tilde{H}^{*,*}(DQ_n), \M_2)$
for $n$ congruent to $0$ modulo $4$.  The interested reader can
easily construct charts for values of $n$ other than $16$.
Similarly,
Figures \ref{fig:DQn=2} and \ref{fig:DQn=1} are representative
of the computation 
of $\Ext^*_{\SA(1)}(\tilde{H}^{*,*}(DQ_n), \M_2)$
for $n$ congruent to $2$ and $1$ modulo $4$.


\begin{figure}[htbp!]
\begin{center}
\setlength{\unitlength}{0.63cm}
\begin{picture}(20,13)(-1,-1)
\put(-1,-1){\line(1,0){20}}
\put(-1,-1){\line(0,1){12}}

\put(0,-1.5){\tiny 0}
\put(1,-1.5){\tiny 1}
\put(4,-1.5){\tiny 4}
\put(8,-1.5){\tiny 8}
\put(12,-1.5){\tiny 12}
\put(16,-1.5){\tiny 16}

\put(0,0){\bc}
\put(0,0){\mhz} \put(0.1,1.3){\tiny $h_0(0)$}
\put(0,1){\bc}
\put(0,1){\mhz}
\put(0,2){\bc}
\put(0,2){\mhz}
\put(0,3){\bc}
\put(0,3){\mhz}
\put(0,4){\bc}
\put(0,4){\mhz}
\put(0,5){\bc}
\put(0,5){\mhz}
\put(0,6){\bc}
\put(0,6){\mhz}
\put(0,7){\bc}
\put(0,7){\mhz}
\put(0,8){\bc}
\put(0,8){\mhz}
\put(0,9){\bc}
\put(0,9){\mhz}
\put(0,10){\bc}
\put(0,10){\mhz}

\put(0,0){\mhone}
\put(1,1){\bc} \put(1,0.4){\tiny $h_1(1)$}
\put(1,1){\mhone}
\put(2,2){\bc}
\put(2,2){\mhone}
\put(3,3){\rc}
\put(3,3){\mhone}
\put(4,4){\rc}
\put(4,4){\mhone}
\put(5,5){\rc}
\put(5,5){\mhone}
\put(6,6){\rc}
\put(6,6){\mhone}
\put(7,7){\rc}
\put(7,7){\mhone}
\put(8,8){\rc}
\put(8,8){\mhone}
\put(9,9){\rc}
\put(9,9){\mhone}
\put(10,10){\rc}
\put(10,10){\mhone}

\put(4,3){\bc}
\put(4,3){\line(1,2){0.5}}
\put(4.5,4){\bc}
\put(4.5,4){\line(-1,2){0.5}}
\put(4,5){\bc}
\put(4,5){\mhz}
\put(4,6){\bc}
\put(4,6){\mhz}
\put(4,7){\bc}
\put(4,7){\mhz}
\put(4,8){\bc}
\put(4,8){\mhz}
\put(4,9){\bc}
\put(4,9){\mhz}
\put(4,10){\bc}
\put(4,10){\mhz}

\put(8,4){\bc}
\put(8,4){\mhz}
\put(8,5){\bc}
\put(8,5){\mhz}
\put(8,6){\bc}
\put(8,6){\mhz}
\put(8,7){\bc}
\put(8,7){\line(1,2){0.5}}
\put(8.5,8){\bc}
\put(8.5,8){\line(-1,2){0.5}}
\put(8,9){\bc}
\put(8,9){\mhz}
\put(8,10){\bc}
\put(8,10){\mhz}

\put(8,4){\mhone}
\put(9,5){\bc}
\put(9,5){\mhone}
\put(10,6){\bc}
\put(10,6){\mhone}
\put(11,7){\rc}
\put(11,7){\mhone}
\put(12,8){\rc}
\put(12,8){\mhone}
\put(13,9){\rc}
\put(13,9){\mhone}
\put(14,10){\rc}
\put(14,10){\mhone}

\put(12,7){\bc}
\put(12,7){\line(1,2){0.5}}
\put(12.5,8){\bc}
\put(12.5,8){\line(-1,2){0.5}}
\put(12,9){\bc}
\put(12,9){\mhz}
\put(12,10){\bc}
\put(12,10){\mhz}

\put(16,8){\bc}
\put(16,8){\mhz}
\put(16,9){\bc}
\put(16,9){\mhz}
\put(16,10){\bc}
\put(16,10){\mhz}

\put(16,8){\mhone} \put(16,7.5){\tiny $\beta ^{2}$}
\put(17,9){\bc}
\put(17,9){\mhone}
\put(18,10){\bc}
\put(18,10){\mhone}

\put(4,2.5){\tiny $\alpha(2)$}
\put(8,3.5){\tiny $\beta(4)$}
\put(12,6.5){\tiny $\alpha \beta$}
\end{picture}\caption{$\Ext_{\SA(1)}^{*}(\M_2,\M_2)$}\label{fig:AdamsM2}
\end{center}
\end{figure}


\begin{center}
\begin{figure}[htbp!]
\setlength{\unitlength}{0.57cm}
\begin{picture}(22,8)(-1,-1)
\put(-1,-1){\line(1,0){22}}
\put(-1,-1){\line(0,1){7}}

\put(0,-1.5){\tiny -1}
\put(4,-1.5){\tiny 3}
\put(8,-1.5){\tiny 7}
\put(12,-1.5){\tiny 11}
\put(16,-1.5){\tiny 15}
\put(20,-1.5){\tiny 19}
\put(0,0){\bc} \put(0,0){\mhz} \put(0,-0.7){\tiny $r_{-1}(0)$}
\put(0,1){\bc} \put(0,1){\mhz}
\put(0,2){\bc}  \put(0,2){\mhz}
\put(0,3){\bc}  \put(0,3){\mhz}
\put(0,4){\bc}  \put(0,4){\mhz}
\put(0,5){\bc}  \put(0,5){\mhz}

\put(4,0){\bc}  \put(4,0){\mhz} \put(4,-0.7){\tiny $r_{3}(2)$}
\put(4,1){\bc}   \put(4,1){\mhz}
\put(4,2){\bc}   \put(4,2){\mhz}
\put(4,3){\bc} \put(4,3){\mhz}
\put(4,4){\bc} \put(4,4){\mhz}
\put(4,5){\bc} \put(4,5){\mhz}

\put(8,0){\bc} \put(8,0){\mhz} \put(8,-0.7){\tiny $r_{7}(4)$}
\put(8,1){\bc}  \put(8,1){\mhz}
\put(8,2){\bc}\put(8,2){\mhz}
\put(8,3){\bc}\put(8,3){\mhz}
\put(8,4){\bc}\put(8,4){\mhz}
\put(8,5){\bc}\put(8,5){\mhz}

\put(12,0){\bc}\put(12,0){\mhz} \put(12,-0.7){\tiny $r_{11}(6)$}
\put(12,1){\bc}\put(12,1){\mhz}
\put(12,2){\bc}\put(12,2){\mhz}
\put(12,3){\bc}\put(12,3){\mhz}
\put(12,4){\bc}\put(12,4){\mhz}
\put(12,5){\bc}\put(12,5){\mhz}

\put(16,0){\bc}\put(16,0){\mhz} \put(16,-0.7){\tiny $r_{15}(8)$}
\put(16,1){\bc}\put(16,1){\mhz}
\put(16,2){\bc}\put(16,2){\mhz}
\put(16,3){\bc}\put(16,3){\mhz}
\put(16,4){\bc}\put(16,4){\mhz}
\put(16,5){\bc}\put(16,5){\mhz}

\put(20,0){\bc}\put(20,0){\mhz} \put(20,-0.7){\tiny $r_{19}(10)$}
\put(20,1){\bc}\put(20,1){\mhz}
\put(20,2){\bc}\put(20,2){\mhz}
\put(20,3){\bc}\put(20,3){\mhz}
\put(20,4){\bc}\put(20,4){\mhz}
\put(20,5){\bc}\put(20,5){\mhz}

\end{picture}
\caption{$\Ext_{\SA(1)}^{*}(R,\M_2)$}\label{fig:Ext(R,M2)}
\end{figure}
\end{center}


\begin{center}
\begin{figure}[htbp!]
\setlength{\unitlength}{0.5cm}
\begin{picture}(25,15)(-1,-1)
\put(-1,-1){\line(1,0){25}}
\put(-1,-1){\line(0,1){14}}

\put(0,-1.5){\tiny 0}
\put(1,-1.5){\tiny 1}
\put(2,-1.5){\tiny 2}
\put(3,-1.5){\tiny 3}
\put(5,-1.5){\tiny 5}
\put(7,-1.5){\tiny 7}
\put(9,-1.5){\tiny 9}
\put(11,-1.5){\tiny 11}
\put(13,-1.5){\tiny 13}
\put(15,-1.5){\tiny 15}
\put(17,-1.5){\tiny 17}
\put(19,-1.5){\tiny 19}
\put(21,-1.5){\tiny 21}
\put(23,-1.5){\tiny 23}
\put(1,0){\bc} \put(0.5,-0.6){\tiny $k(1)$}
\put(2,1){\bc}
\put(3,2){\bc}
\put(4,3){\rc}
\put(5,4){\rc}
\put(6,5){\rc}
\put(7,6){\rc}
\put(8,7){\rc}
\put(9,8){\rc}
\put(10,9){\rc}
\put(11,10){\rc}
\put(12,11){\rc}
\put(13,12){\rc}

\put(9,4){\bc} \put(8.5,3.4){\tiny $\beta k$}
\put(10,5){\bc}
\put(11,6){\bc}
\put(12,7){\rc}
\put(13,8){\rc}
\put(14,9){\rc}
\put(15,10){\rc}
\put(16,11){\rc}
\put(17,12){\rc}

\put(17,8){\bc}\put(16.5,7.3){\tiny $\beta^2 k$}
\put(18,9){\bc}
\put(19,10){\bc}
\put(20,11){\rc}
\put(21,12){\rc}

\put(1,0){\mhone}
\put(2,1){\mhone}
\put(3,2){\mhone}
\put(4,3){\mhone}
\put(5,4){\mhone}
\put(6,5){\mhone}
\put(7,6){\mhone}
\put(8,7){\mhone}
\put(9,8){\mhone}
\put(10,9){\mhone}
\put(11,10){\mhone}
\put(12,11){\mhone}
\put(13,12){\mhone}

\put(9,4){\mhone}
\put(10,5){\mhone}
\put(11,6){\mhone}
\put(12,7){\mhone}
\put(13,8){\mhone}
\put(14,9){\mhone}
\put(15,10){\mhone}
\put(16,11){\mhone}
\put(17,12){\mhone}

\put(17,8){\mhone}
\put(18,9){\mhone}
\put(19,10){\mhone}
\put(20,11){\mhone}
\put(21,12){\mhone}

\put(3,0){\bc} \put(2.5,-0.6){\tiny $r_{3}(2)$}
\put(3,1){\bc}
\put(3,2){\bc}
\put(7,0){\bc} \put(6.5,-0.6){\tiny $r_{7}(4)$}
\put(7,1){\bc}
\put(7,2){\bc}
\put(7,3){\bc}
\put(11,0){\bc} \put(10.5,-0.6){\tiny $r_{11}(6)$} 
\put(11,1){\bc}
\put(11,2){\bc}
\put(11,3){\bc}
\put(11,4){\bc}
\put(11,5){\bc}
\put(11,6){\bc}
\put(15,0){\bc} \put(14.5,-0.6){\tiny $r_{15}(8)$}
\put(15,1){\bc}
\put(15,2){\bc}
\put(15,3){\bc}
\put(15,4){\bc}
\put(15,5){\bc}
\put(15,6){\bc}
\put(15,7){\bc}
\put(19,0){\bc} \put(18.5,-0.6){\tiny $r_{19}(10)$}
\put(19,1){\bc}
\put(19,2){\bc}
\put(19,3){\bc}
\put(19,4){\bc}
\put(19,5){\bc}
\put(19,6){\bc}
\put(19,7){\bc}
\put(19,8){\bc}
\put(19,9){\bc}
\put(19,10){\bc}
\put(23,0){\bc} \put(22.5,-0.6){\tiny $r_{23}(12)$}
\put(23,1){\bc}
\put(23,2){\bc}
\put(23,3){\bc}
\put(23,4){\bc}
\put(23,5){\bc}
\put(23,6){\bc}
\put(23,7){\bc}
\put(23,8){\bc}
\put(23,9){\bc}
\put(23,10){\bc}
\put(23,11){\bc}
\put(3,0){\mhz}
\put(3,1){\mhzt}
\put(7,0){\mhz}
\put(7,1){\mhz}
\put(7,2){\mhz}
\put(11,0){\mhz}
\put(11,1){\mhz}
\put(11,2){\mhz}
\put(11,3){\mhz}
\put(11,4){\mhz}
\put(11,5){\mhzt}
\put(15,0){\mhz}
\put(15,1){\mhz}
\put(15,2){\mhz}
\put(15,3){\mhz}
\put(15,4){\mhz}
\put(15,5){\mhz}
\put(15,6){\mhz}
\put(19,0){\mhz}
\put(19,1){\mhz}
\put(19,2){\mhz}
\put(19,3){\mhz}
\put(19,4){\mhz}
\put(19,5){\mhz}
\put(19,6){\mhz}
\put(19,7){\mhz}
\put(19,8){\mhz}
\put(19,9){\mhzt}
\put(23,0){\mhz}
\put(23,1){\mhz}
\put(23,2){\mhz}
\put(23,3){\mhz}
\put(23,4){\mhz}
\put(23,5){\mhz}
\put(23,6){\mhz}
\put(23,7){\mhz}
\put(23,8){\mhz}
\put(23,9){\mhz}
\put(23,10){\mhz}
\end{picture}
\caption{$\Ext_{\SA(1)}^{*}(\widetilde H^{*,*}(DQ_{\infty}),\M_2)$}
\label{fig:ExtDQinf}
\end{figure}
\end{center}

\newpage


\begin{figure}[htbp!]
\setlength{\unitlength}{0.5cm}
\begin{picture}(25,14)(0,-1)

\put(0,-1){\line(1,0){25}}
\put(0,-1){\line(0,1){13}}
\put(0,-1.5){\tiny 0}
\put(3,-1.5){\tiny 3}
\put(7,-1.5){\tiny 7}
\put(11,-1.5){\tiny 11}
\put(15,-1.5){\tiny 15}
\put(19,-1.5){\tiny 19}
\put(23,-1.5){\tiny 23}

\put(1,0){\bc}\put(1,0){\mhone}\put(0.5,-0.6){\tiny $k(1)$}
\put(2,1){\bc}\put(2,1){\mhone}
\put(3,2){\bc}\put(3,2){\mhone}
\put(4,3){\rc}\put(4,3){\mhone}
\put(5,4){\rc}\put(5,4){\mhone}
\put(6,5){\rc}\put(6,5){\mhone}
\put(7,6){\rc}\put(7,6){\mhone}
\put(8,7){\rc}\put(8,7){\mhone}
\put(9,8){\rc}\put(9,8){\mhone}
\put(10,9){\rc}\put(10,9){\mhone}
\put(11,10){\rc}\put(11,10){\mhone}
\put(12,11){\rc}\put(12,11){\mhone}

\put(3,0){\bc}\put(3,0){\mhz} \put(2.5,-0.6){\tiny $r_3(2)$}
\put(3,1){\bc}\put(3,1){\mhzt}

\put(7,0){\bc}\put(7,0){\mhz} \put(6.5,-0.6){\tiny $r_7(4)$}
\put(7,1){\bc}\put(7,1){\mhz}
\put(7,2){\bc}\put(7,2){\mhz}
\put(7,3){\bc}

\put(9,4){\bc}\put(9,4){\mhone} \put(8.5,3.4){\tiny $\beta k$}
\put(10,5){\bc}\put(10,5){\mhone}
\put(11,6){\bc}\put(11,6){\mhone}
\put(12,7){\rc}\put(12,7){\mhone}
\put(13,8){\rc}\put(13,8){\mhone}
\put(14,9){\rc}\put(14,9){\mhone}
\put(15,10){\rc}\put(15,10){\mhone}
\put(16,11){\rc}\put(16,11){\mhone}

\put(11,0){\bc}\put(11,0){\mhz} \put(10.5,-0.6){\tiny $r_{11}(6)$}
\put(11,1){\bc}\put(11,1){\mhz}
\put(11,2){\bc}\put(11,2){\mhz}
\put(11,3){\bc}\put(11,3){\mhz}
\put(11,4){\bc}\put(11,4){\mhz}
\put(11,5){\bc}\put(11,5){\mhzt}

\put(15,0){\bc}\put(15,0){\mhz} \put(14.5,-0.6){\tiny $r_{15}(8)$}
\put(15,1){\bc}\put(15,1){\mhz}
\put(15,2){\bc}\put(15,2){\mhz}
\put(15,3){\bc}\put(15,3){\mhz}
\put(15,4){\bc}\put(15,4){\mhz}
\put(15,5){\bc}\put(15,5){\mhz}
\put(15,6){\bc}\put(15,6){\mhz}
\put(15,7){\bc}

\put(15,0){\mhone}

\put(17,8){\bc}\put(17,8){\mhone} \put(16.5,7.4){\tiny $\beta^2 k$}
\put(18,9){\bc}\put(18,9){\mhone}
\put(19,10){\bc}\put(19,10){\mhone}
\put(20,11){\rc}\put(20,11){\mhone}

\put(19,3){\bc}
\put(19,3){\line(1,2){0.5}} \put(18.5,2.4){\tiny $\alpha r_{15}$}
\put(19.5,4){\bc}
\put(19.5,4){\line(-1,2){0.5}}
\put(19,5){\bc}\put(19,5){\mhz}
\put(19,6){\bc}\put(19,6){\mhz}
\put(19,7){\bc}\put(19,7){\mhz}
\put(19,8){\bc}\put(19,8){\mhz}
\put(19,9){\bc}\put(19,9){\mhzt}
\put(19,10){\bc}

\put(23,4){\bc}
\put(23,4){\mhz}
\put(23,4){\mhone}
  \put(22.5,3.4){\tiny $\beta r_{15}$}
\put(23,5){\bc}\put(23,5){\mhz}
\put(23,6){\bc}\put(23,6){\mhz}
\put(23,7){\bc}\put(23,7){\line(1,2){0.5}}
\put(23.5,8){\bc}\put(23.5,8){\line(-1,2){0.5}}
\put(23,9){\bc}\put(23,9){\mhz}
\put(23,10){\bc}\put(23,10){\mhz}
\put(23,11){\bc}

\put(16,1){\bc} \put(16,1){\mhone}
\put(17,2){\bc}\put(17,2){\mhone}
\put(18,3){\rc}\put(18,3){\mhone}
\put(19,4){\rc}\put(19,4){\mhone}
\put(20,5){\rc}\put(20,5){\mhone}
\put(21,6){\rc}\put(21,6){\mhone}
\put(22,7){\rc}\put(22,7){\mhone}
\put(23,8){\rc}\put(23,8){\mhone}

\put(24,5){\bc}\put(24,5){\mhone}

\end{picture}
\caption{$\Ext_{\SA(1)}^{*}(\widetilde H^{*,*}(DQ_{16}),\M_2)$
\label{fig:DQn=0}}
\end{figure}


\begin{center}
\begin{figure}[htbp!]
\setlength{\unitlength}{0.7cm}
\begin{picture}(18,11)(0,-1)
\put(0,-1){\line(1,0){18}}
\put(0,-1){\line(0,1){10}}

\put(0,-1.5){\tiny 0}
\put(1,-1.5){\tiny 1}
\put(3,-1.5){\tiny 3}
\put(9,-1.5){\tiny 9}
\put(11,-1.5){\tiny 11}

\put(1,0){\bc}\put(1,0){\mhone} \put(0.5,-0.6){\tiny $x(1)$}
\qbezier[20](1,0)(4,2.25)(5,3)
\put(2,1){\bc}\put(2,1){\mhone}
\put(3,2){\bc}\put(3,2){\mhone} 
\put(4,3){\rc}\put(4,3){\mhone}
\put(5,4){\rc}\put(5,4){\mhone}
\put(6,5){\rc}\put(6,5){\mhone}
\put(7,6){\rc}\put(7,6){\mhone}
\put(8,7){\rc}\put(8,7){\mhone}
\put(9,8){\rc}\put(9,8){\mhone}

\put(3,1){\bc}\put(3,1){\mhone}\put(3,1){\mhzt} \put(2.5,0.4){\tiny $y(2)$}
\put(4,2){\bc}\put(4,2){\mhone}
\put(5,3){\bc}\put(5,3){\mhone}
\put(6,4){\rc}\put(6,4){\mhone}
\put(7,5){\rc}\put(7,5){\mhone}
\put(8,6){\rc}\put(8,6){\mhone}
\put(9,7){\rc}\put(9,7){\mhone}
\put(10,8){\rc}\put(10,8){\mhone}

\put(9,4){\bc}\put(9,4){\mhone} \put(8.5,3.4){\tiny $\beta x$}
\qbezier[20](9,4)(12,6.25)(13,7)
\put(10,5){\bc}\put(10,5){\mhone}
\put(11,6){\bc}\put(11,6){\mhone}
\put(12,7){\rc}\put(12,7){\mhone}
\put(13,8){\rc}\put(13,8){\mhone}

\put(11,5){\bc}\put(11,5){\mhone}\put(11,5){\mhzt} 
  \put(10.5,4.4){\tiny $\beta y$}
\put(12,6){\bc}\put(12,6){\mhone}
\put(13,7){\bc}\put(13,7){\mhone}
\put(14,8){\rc}\put(14,8){\mhone}

\end{picture}
\caption{Adams Chart for $\Ext_{\SA(1)}^{*}(\widetilde H^{*,*}(DQ_2), \M_2)$}
\label{fig:DQ2}
\end{figure}
\end{center}

\vspace{5in}
\mbox{}

\newpage


\begin{center}
\begin{figure}[htbp!]
\setlength{\unitlength}{0.5cm}
\begin{picture}(25,14)(0,-1)
\put(0,-1){\line(1,0){25}}
\put(0,-1){\line(0,1){13}}

\put(0,-1.5){\tiny 0}
\put(1,-1.5){\tiny 1}
\put(3,-1.5){\tiny 3}
\put(7,-1.5){\tiny 7}
\put(11,-1.5){\tiny 11}
\put(15,-1.5){\tiny 15}
\put(19,-1.5){\tiny $19$}
\put(23,-1.5){\tiny $23$}

\put(1,0){\bc}\put(1,0){\mhone} \put(0.5,-0.6){\tiny $k(1)$}
\put(2,1){\bc}\put(2,1){\mhone}
\put(3,2){\bc}\put(3,2){\mhone}
\put(4,3){\rc}\put(4,3){\mhone}
\put(5,4){\rc}\put(5,4){\mhone}
\put(6,5){\rc}\put(6,5){\mhone}
\put(7,6){\rc}\put(7,6){\mhone}
\put(8,7){\rc}\put(8,7){\mhone}
\put(9,8){\rc}\put(9,8){\mhone}
\put(10,9){\rc}\put(10,9){\mhone}
\put(11,10){\rc}\put(11,10){\mhone}
\put(12,11){\rc}\put(12,11){\mhone}

\put(3,0){\bc}\put(3,0){\mhz} \put(2.5,-0.6){\tiny $r_3(2)$}
\put(3,1){\bc}\put(3,1){\mhzt}

\put(7,0){\bc}\put(7,0){\mhz} \put(6.5,-0.6){\tiny $r_7(4)$}
\put(7,1){\bc}\put(7,1){\mhz}
\put(7,2){\bc}\put(7,2){\mhz}
\put(7,3){\bc}

\put(9,4){\bc}\put(9,4){\mhone} \put(8.5,3.4){\tiny $\beta k$}
\put(10,5){\bc}\put(10,5){\mhone}
\put(11,6){\bc}\put(11,6){\mhone}
\put(12,7){\rc}\put(12,7){\mhone}
\put(13,8){\rc}\put(13,8){\mhone}
\put(14,9){\rc}\put(14,9){\mhone}
\put(15,10){\rc}\put(15,10){\mhone}
\put(16,11){\rc}\put(16,11){\mhone}

\put(11,0){\bc}\put(11,0){\mhz} \put(10.5,-0.6){\tiny $r_{11}(6)$}
\put(11,1){\bc}\put(11,1){\mhz}
\put(11,2){\bc}\put(11,2){\mhz}
\put(11,3){\bc}\put(11,3){\mhz}
\put(11,4){\bc}\put(11,4){\mhz}
\put(11,5){\bc}\put(11,5){\mhzt}

\put(15,1){\bc}\put(15,1){\mhz} 
  \put(14.5,0.2){\tiny $z(8)$}
\put(15,2){\bc}\put(15,2){\mhz}
\put(15,3){\bc}\put(15,3){\mhz}
\put(15,4){\bc}\put(15,4){\mhz}
\put(15,5){\bc}\put(15,5){\mhz}
\put(15,6){\bc}\put(15,6){\mhz}
\put(15,7){\bc}

\put(15,1){\mhone}
\put(16,2){\bc}\put(16,2){\mhone}
\put(17,3){\bc}\put(17,3){\mhone}
\put(18,4){\rc}\put(18,4){\mhone}
\put(19,5){\rc}\put(19,5){\mhone}
\put(20,6){\rc}\put(20,6){\mhone}
\put(21,7){\rc}\put(21,7){\mhone}
\put(22,8){\rc}\put(22,8){\mhone}
\put(23,9){\rc}\put(23,9){\mhone}

\put(17,8){\bc}\put(17,8){\mhone} \put(16.5,7.4){\tiny $\beta^2 k$}
\put(18,9){\bc}\put(18,9){\mhone}
\put(19,10){\bc}\put(19,10){\mhone}
\put(20,11){\rc}\put(20,11){\mhone}

\put(19,4){\bc}\put(19,4){\line(1,2){0.5}} \put(18.5,3.4){\tiny $\alpha z$}
\put(19.5,5){\bc}\put(19.5,5){\line(-1,2){0.5}}
\put(19,6){\bc}\put(19,6){\mhz}
\put(19,7){\bc}\put(19,7){\mhz}
\put(19,8){\bc}\put(19,8){\mhz}
\put(19,9){\bc}\put(19,9){\mhzt}

\put(23,5){\bc}\put(23,5){\mhz} \put(22.5,4.4){\tiny $\beta z$}
\put(23,6){\bc}\put(23,6){\mhz}
\put(23,7){\bc}\put(23,7){\mhz}
\put(23,8){\bc}\put(23,8){\line(1,2){0.5}}
\put(23.5,9){\bc}\put(23.5,9){\line(-1,2){0.5}}
\put(23,10){\bc}\put(23,10){\mhz}
\put(23,11){\bc}

\put(23,5){\mhone}
\put(24,6){\bc}\put(24,6){\mhone}

\end{picture}
\caption{$\Ext_{\SA(1)}^{*}(\widetilde H^{*,*}(DQ_{14}), \M_2)$ 
\label{fig:DQn=2}}
\end{figure}
\end{center}


\begin{center}
\begin{figure}[htbp!]
\setlength{\unitlength}{0.52cm}
\begin{picture}(24,15)(0,-1)
\put(0,-1){\line(1,0){24}}
\put(0,-1){\line(0,1){14}}

\put(0,-1.5){\tiny 0}
\put(1,-1.5){\tiny 1}
\put(3,-1.5){\tiny 3}
\put(7,-1.5){\tiny 7}
\put(11,-1.5){\tiny 11}
\put(15,-1.5){\tiny $15$}
\put(17,-1.5){\tiny $17$}
\put(19,-1.5){\tiny $19$}
\put(21,-1.5){\tiny $21$}
\put(23,-1.5){\tiny $23$}

\put(1,0){\bc}\put(1,0){\mhone}\put(0.5,-0.6){\tiny $k(1)$}
\put(2,1){\bc}\put(2,1){\mhone}
\put(3,2){\bc}\put(3,2){\mhone}
\put(4,3){\rc}\put(4,3){\mhone}
\put(5,4){\rc}\put(5,4){\mhone}
\put(6,5){\rc}\put(6,5){\mhone}
\put(7,6){\rc}\put(7,6){\mhone}
\put(8,7){\rc}\put(8,7){\mhone}
\put(9,8){\rc}\put(9,8){\mhone}
\put(10,9){\rc}\put(10,9){\mhone}
\put(11,10){\rc}\put(11,10){\mhone}
\put(12,11){\rc}\put(12,11){\mhone}
\put(13,12){\rc}\put(13,12){\mhone}

\put(3,0){\bc}\put(3,0){\mhz}\put(2.5,-0.6){\tiny $r_{3}(2)$}
\put(3,1){\bc}\put(3,1){\mhzt}

\put(7,0){\bc}\put(7,0){\mhz}\put(6.5,-0.6){\tiny $r_{7}(4)$}
\put(7,1){\bc}\put(7,1){\mhz}
\put(7,2){\bc}\put(7,2){\mhz}
\put(7,3){\bc}

\put(9,4){\bc}\put(9,4){\mhone} \put(8.5,3.4){\tiny $\beta k$}
\put(10,5){\bc}\put(10,5){\mhone}
\put(11,6){\bc}\put(11,6){\mhone}
\put(12,7){\rc}\put(12,7){\mhone}
\put(13,8){\rc}\put(13,8){\mhone}
\put(14,9){\rc}\put(14,9){\mhone}
\put(15,10){\rc}\put(15,10){\mhone}
\put(16,11){\rc}\put(16,11){\mhone}
\put(17,12){\rc}\put(17,12){\mhone}

\put(11,0){\bc}\put(11,0){\mhz}\put(10.5,-0.6){\tiny $r_{11}(6)$}
\put(11,1){\bc}\put(11,1){\mhz}
\put(11,2){\bc}\put(11,2){\mhz}
\put(11,3){\bc}\put(11,3){\mhz}
\put(11,4){\bc}\put(11,4){\mhz}
\put(11,5){\bc}\put(11,5){\mhzt}

\put(15,0){\bc}\put(15,0){\mhz}\put(14.5,-0.6){\tiny $r_{15}(8)$}
\put(15,1){\bc}\put(15,1){\mhz}
\put(15,2){\bc}\put(15,2){\mhz}
\put(15,3){\bc}\put(15,3){\mhz}
\put(15,4){\bc}\put(15,4){\mhz}
\put(15,5){\bc}\put(15,5){\mhz}
\put(15,6){\bc}\put(15,6){\mhz}
\put(15,7){\bc}

\put(17,1){\bc}\put(17,1){\mhz}
  \put(16.1,0.2){\tiny $v(9)$}
\put(17,2){\bc}\put(17,2){\mhz}
\put(17,3){\bc}\put(17,3){\mhz}
\put(17,4){\bc}\put(17,4){\mhz}
\put(17,5){\bc}\put(17,5){\mhz}
\put(17,6){\bc}\put(17,6){\mhz}
\put(17,7){\bc}\put(17,7){\line(-1,2){0.5}}
\put(16.5,8){\bc}\put(16.5,8){\line(1,2){0.5}}
\put(17,9){\bc}\put(17,9){\mhz}
\put(17,10){\bc}\put(17,10){\mhz}
\put(17,11){\bc}\put(17,11){\line(1,2){0.5}}
\put(17.5,12){\bc}\put(17.5,12){\line(-1,2){0.5}}

\put(17,8){\bc}\put(17,8){\mhone} \put(17.2,7.4){\tiny $\beta^2 k$}
\put(18,9){\bc}\put(18,9){\mhone}
\put(19,10){\bc}\put(19,10){\mhone}
\put(20,11){\rc}\put(20,11){\mhone}
\put(21,12){\rc}\put(21,12){\mhone}

\put(19,2){\bc}\put(19,2){\mhz}
  \put(18.1,1.2){\tiny $u(10)$}
\put(19,3){\bc}\put(19,3){\mhz}
\put(19,4){\bc}\put(19,4){\mhz}
\put(19,5){\bc}\put(19,5){\mhz}
\put(19,6){\bc}\put(19,6){\mhz}
\put(19,7){\bc}\put(19,7){\mhz}
\put(19,8){\bc}\put(19,8){\mhz}
\put(19,9){\bc}\put(19,9){\mhzt}

\put(21,3){\bc}\put(21,3){\mhz} 
  \put(20.1,2.2){\tiny $t(11)$}
\put(21,4){\bc}\put(21,4){\mhz}
\put(21,5){\bc}\put(21,5){\mhz}
\put(21,6){\bc}\put(21,6){\mhz}
\put(21,7){\bc}\put(21,7){\mhz}
\put(21,8){\bc}\put(21,8){\mhz}
\put(21,9){\bc}\put(21,9){\mhz}
\put(21,10){\bc}\put(21,10){\mhz}
\put(21,11){\bc}\put(21,11){\line(1,2){0.5}}
\put(21.5,12){\bc}\put(21.5,12){\line(-1,2){0.5}}

\put(23,4){\bc}\put(23,4){\mhz} 
  \put(22.3,3.4){\tiny $\beta r_{15}$}
\put(23,5){\bc}\put(23,5){\mhz}
\put(23,6){\bc}\put(23,6){\mhz}
\put(23,7){\bc}\put(23,7){\mhz}
\put(23,8){\bc}\put(23,8){\mhz}
\put(23,9){\bc}\put(23,9){\mhz}
\put(23,10){\bc}\put(23,10){\mhz}
\put(23,11){\bc}

\end{picture}
\caption{$\Ext_{\SA(1)}^{*}(\widetilde H^{*,*}(DQ_{17}),\M_2)$ 
\label{fig:DQn=1}}
\end{figure}
\end{center}

\vspace{5in}
\mbox{}


\end{document}